\documentclass[a4paper, 11pt]{amsart}

\usepackage{amsmath, amssymb, amsthm, mathtools}
\usepackage[colorlinks=true, linkcolor=blue]{hyperref}

\usepackage[margin=1in]{geometry}
\usepackage{quiver}
\usepackage{booktabs}
\usepackage{array}
\usepackage{float}

\newtheorem{theorem}{Theorem}[section]
\newtheorem{lemma}[theorem]{Lemma}
\newtheorem{proposition}[theorem]{Proposition}
\newtheorem{corollary}[theorem]{Corollary}

\newtheorem*{maintheoremA}{Main Theorem A}
\newtheorem*{maintheoremB}{Main Theorem B}
\newtheorem*{maintheoremC}{Main Theorem C}

\theoremstyle{definition}
\newtheorem{definition}[theorem]{Definition}

\newtheorem{remark}[theorem]{Remark}
\newtheorem{question}[theorem]{Question}

\newcommand{\NN}{\mathbb{N}}

\newcommand{\QQ}{\mathbb{Q}}

\newcommand{\PP}{\mathbb{P}}

\newcommand{\NC}{\textnormal{(NC)}}
\newcommand{\QC}{\textnormal{(QC)}}

\newcommand{\tr}{\operatorname{tr}}

\newcommand{\embdim}{\operatorname{edim}}
\newcommand{\embcodim}{\operatorname{codim}}
\newcommand{\Spec}{\operatorname{Spec}}
\newcommand{\height}{\operatorname{ht}}

\newcommand{\rank}{\operatorname{rank}}
\newcommand{\depth}{\operatorname{depth}}
\newcommand{\Hom}{\operatorname{Hom}}
\newcommand{\Ass}{\operatorname{Ass}}
\newcommand{\Tor}{\operatorname{Tor}}
\newcommand{\Ext}{\operatorname{Ext}}
\newcommand{\projdim}{\operatorname{proj.dim}}

\renewcommand{\hat}[1]{\widehat{#1}}

\title[Failure of the Nagata Criterion and Recovery of Openness]{Complete Intersections of Bounded Codimension: Failure of the Nagata Criterion and Recovery of Openness}
\author{Rirai Ikeda}
\address{Department of Mathematics, School of Science, Science Tokyo, 2-12-1 Ookayama, Meguro-ku, Tokyo 152-8551, Japan}
\email{ikeda.r.cfbf@m.isct.ac.jp}

\begin{document}

\begin{abstract}
    Let $\mathsf{CI}_{\leq c}$ denote the property of being a complete intersection of codimension at most $c$. Although the regular, complete intersection, Gorenstein, and Cohen--Macaulay properties satisfy the Nagata criterion \NC{}, we prove that $\mathsf{CI}_{\leq c}$ does not satisfy \NC{} for any $c \geq 1$. We first construct a counterexample for hypersurfaces and then obtain counterexamples for arbitrary $c$ using square-zero extensions. We also introduce three conditions for a property of Noetherian local rings and show that, under stability with respect to localization and reduction by suitable regular sequences, \NC{} is characterized by a lifting property across normally flat nilpotent thickenings. Nevertheless, we recover the expected openness result: the $\mathsf{CI}_{\leq c}$-locus is open for every Noetherian ring satisfying $\mathsf{Reg}$-Q0, and hence for every quasi-excellent ring. Finally, for every quasi-compact excellent scheme $X$, we prove that the subset 
    \begin{align*}
        \left\{x \in X \mid \embdim(\mathcal{O}_{X,x}) - \dim(\mathcal{O}_{X,x}) \leq n \right\}
    \end{align*}
    is constructible for every $n \in \NN$, although the function $x \mapsto \embdim(\mathcal{O}_{X,x}) - \dim(\mathcal{O}_{X,x})$ is not upper semicontinuous in general.
\end{abstract}

\maketitle

\tableofcontents

\section{Introduction}\label{sec:introduction}

Throughout this paper, all rings are assumed to be commutative with identity and $\NN = \{0, 1, 2, \ldots\}$.

Let $\PP$ be a property of Noetherian local rings. For a Noetherian ring $A$, one considers the $\PP$-locus
\begin{align*}
    \PP(A) \coloneqq \{ \mathfrak{p} \in \Spec(A) \mid A_{\mathfrak{p}} \text{ satisfies } \PP \}.
\end{align*}
A fundamental problem in commutative algebra and algebraic geometry is to determine when $\PP(A)$ is Zariski open in $\Spec(A)$. A basic principle underlying many classical openness results is the following Nagata criterion, \NC{}.

\begin{enumerate}
    \item[(NC)] Let $A$ be a Noetherian ring. If $\PP(A/\mathfrak{p})$ contains a nonempty open subset of $\Spec(A/\mathfrak{p})$ for every $\mathfrak{p} \in \Spec(A)$, then $\PP(A)$ is an open subset of $\Spec(A)$.
\end{enumerate}

We write $\mathsf{Reg}$, $\mathsf{CI}$, $\mathsf{Gor}$, and $\mathsf{CM}$ for the properties of being regular, complete intersection, Gorenstein, and Cohen--Macaulay, respectively. Nagata proved that $\mathsf{Reg}$ satisfies \NC{} in \cite{Nag59}. The cases of $\mathsf{CI}$ and $\mathsf{Gor}$ were established by Greco and Marinari in \cite{GM78}, and the case of $\mathsf{CM}$ was established by Massaza and Valabrega in \cite{MV77}. Thus, \NC{} holds for all four of these standard classes of singularities. Several other properties are also known to satisfy \NC{}; see Definition \ref{def:Nagata_criterion} and Remark \ref{rem:NC} for the precise definition and further examples. The Nagata criterion has also been extended to loci of finitely generated modules; see Kimura \cite{Kim23}.

In this paper, we study the property $\mathsf{CI}_{\leq c}$ of being a complete intersection of codimension at most $c$, where $c \in \NN$; see Definition \ref{def:cicodim}. In particular, $\mathsf{CI}_{\leq 0}=\mathsf{Reg}$, and $\mathsf{CI}_{\leq 1}=\mathsf{HS}$ is the property of being a hypersurface. These properties fit into the following hierarchy:
\begin{align*}
    \mathsf{Reg} = \mathsf{CI}_{\leq 0} \Longrightarrow \mathsf{HS} = \mathsf{CI}_{\leq 1} \Longrightarrow \mathsf{CI}_{\leq 2} \Longrightarrow \cdots \Longrightarrow \mathsf{CI}_{\leq c} \Longrightarrow \cdots \Longrightarrow \mathsf{CI} \Longrightarrow \mathsf{Gor} \Longrightarrow \mathsf{CM}.
\end{align*}

The properties $\mathsf{CI}_{\leq c}$ are natural subclasses of complete intersections which are closer to regularity as the bound $c$ becomes smaller. Moreover, they retain many of the standard properties enjoyed by the familiar classes above. For example, they are stable under localization and polynomial extensions and have good behavior with respect to reduction by suitable regular sequences. We establish these and related facts in the first subsection of Section \ref{sec:ci-bounded}. Since $\mathsf{Reg}$, $\mathsf{CI}$, $\mathsf{Gor}$, and $\mathsf{CM}$ all satisfy \NC{}, these observations give substantial reason to expect the same conclusion for $\mathsf{CI}_{\leq c}$.

Contrary to this expectation, $\mathsf{CI}_{\leq c}$ does not satisfy \NC{} for any positive value of $c$. This is our first main result.

\begin{maintheoremA}[{Theorem \ref{th:counterexampleNC}}]
    For every $c \geq 1$, the property $\mathsf{CI}_{\leq c}$ does not satisfy \NC{}.
\end{maintheoremA}

We first construct, using an example of Nishimura \cite{Nis12}, a Noetherian ring satisfying $\mathsf{HS}$-Q0 whose hypersurface locus is not open. This gives a counterexample to \NC{} for $\mathsf{HS}=\mathsf{CI}_{\leq 1}$. We then obtain counterexamples for arbitrary $c\geq 1$ by using square-zero extensions of the form
\begin{align*}
    A[t]/(t^2).
\end{align*}
Such an extension leaves the underlying spectrum unchanged and increases the codimension of every local complete intersection by one.

To analyze this failure systematically, we introduce three conditions, denoted by (A-1), (A-2), and (A-3), for a property $\PP$ of Noetherian local rings; see Definition \ref{def:conditions-for-NC}. These conditions concern stability under localization, reduction by a suitable regular sequence, and lifting across a normally flat nilpotent thickening, respectively. We prove that (A-1), (A-2), and (A-3) together imply \NC{}, while \NC{} always implies (A-3); see Proposition \ref{prop:conditions-imply-NC}. The property $\mathsf{CI}_{\leq c}$ satisfies (A-1) and (A-2), but does not satisfy (A-3). Thus, its failure to satisfy \NC{} is accounted for by the failure of the lifting condition across normally flat nilpotent thickenings.

The failure of \NC{} does not by itself imply that the $\mathsf{CI}_{\leq c}$-locus behaves poorly for the rings that commonly arise in algebra and geometry. Indeed, the expected openness result remains valid under a condition considerably weaker than excellence. We say that a Noetherian ring $A$ satisfies $\mathsf{Reg}$-Q0 if $\mathsf{Reg}(A/\mathfrak{p})$ contains a nonempty open subset of $\Spec(A/\mathfrak{p})$ for every prime ideal $\mathfrak{p}$ of $A$; see Definition \ref{def:openness}.

Our second main result establishes the openness of the $\mathsf{CI}_{\leq c}$-locus for every ring satisfying this condition.

\begin{maintheoremB}[{Theorem \ref{th:openness_hs}}]
    Let $A$ be a Noetherian ring which satisfies $\mathsf{Reg}$\textup{-Q0}. Then, $\mathsf{CI}_{\leq c}(A)$ is an open subset of $\Spec(A)$ for every $c \in \NN$.
\end{maintheoremB}

In particular, the $\mathsf{CI}_{\leq c}$-locus is open for every excellent ring and every $c \in \NN$, since every excellent ring satisfies $\mathsf{Reg}$-Q0. More generally, the same conclusion holds for every quasi-excellent ring. Therefore, although $\mathsf{CI}_{\leq c}$ does not satisfy \NC{}, the openness conclusion expected from the classical applications of \NC{} remains valid for the substantially broader class of $\mathsf{Reg}$-Q0 rings.

The bound occurring in $\mathsf{CI}_{\leq c}$ also leads naturally to the study of the codimension function. For a locally Noetherian scheme $X$, consider
\begin{align*}
    C \colon X \longrightarrow \NN, \qquad x \longmapsto \embdim(\mathcal{O}_{X,x})-\dim(\mathcal{O}_{X,x}).
\end{align*}
On the complete intersection locus, the inequality $C(x)\leq c$ characterizes the points at which the local ring is a complete intersection of codimension at most $c$. We investigate the topological behavior of this function.

The codimension function is not upper semicontinuous in general. In fact, upper semicontinuity can fail even when $X$ is an affine excellent scheme. Nevertheless, on a quasi-compact excellent scheme, the function retains the weaker but useful property of constructibility.

We prove the following result.

\begin{maintheoremC}[{Theorem \ref{th:constructible}}]
    Let $(X,\mathcal{O}_X)$ be a quasi-compact excellent scheme. Then, for the function
    \begin{align*}
        C \colon X \longrightarrow \NN, \qquad x \longmapsto \embdim(\mathcal{O}_{X,x})-\dim(\mathcal{O}_{X,x})
    \end{align*}
    on $X$, the subset
    \begin{align*}
        C^{-1}(\NN_{\leq n}) = \{x\in X\mid C(x)\leq n\}
    \end{align*}
    is constructible in $X$ for every $n \in \NN$.
\end{maintheoremC}

This theorem describes the variation of codimension independently of the Nagata criterion. As an application, combining the constructibility of $C$ with the stability of $\mathsf{CI}_{\leq c}$ under generalization gives another proof that the $\mathsf{CI}_{\leq c}$-locus is open on every excellent scheme. Thus, the openness result can also be viewed as a consequence of the constructible behavior of the codimension function on the complete intersection locus.

The organization of this paper is as follows. In Section \ref{sec:preliminaries}, we recall basic definitions and results concerning singularity loci. In Section \ref{sec:conditions}, we introduce several conditions for properties of Noetherian local rings, including \NC{} and the quotient condition. In Section \ref{sec:conditions-for-NC}, we establish a criterion for \NC{} in terms of the conditions (A-1), (A-2), and (A-3). In Section \ref{sec:ci-bounded}, we first study the basic properties of complete intersections of bounded codimension, then construct counterexamples to \NC{}, and finally prove the openness of the $\mathsf{CI}_{\leq c}$-locus for $\mathsf{Reg}$-Q0 rings. In Section \ref{sec:constructible}, we study the codimension function on excellent schemes, prove its constructibility, and derive the openness result for excellent schemes from this perspective. Finally, in Section \ref{sec:openq}, we discuss several open questions.

\medskip

\noindent
\textbf{Acknowledgments.} The author would like to thank Taiga Ozaki and Kanau Shimada for helpful discussions during the early stages of this work. The author is also grateful to Kazuki Hayashi and Ryoma Takeuchi for their comments, and to Kazuki Hayashi for valuable comments on the references. Finally, the author would like to express his sincere gratitude to his advisor, Kazuma Shimomoto, for his guidance and support.

\section{Preliminaries}\label{sec:preliminaries}

Let $\PP$ be a property of local rings. For example, $\PP$ can be the property of being regular, being a complete intersection, being Gorenstein, or being Cohen--Macaulay. For any Noetherian ring $A$, we put
\begin{align*}
    \PP(A) \coloneqq \{ \mathfrak{p} \in \Spec(A) \mid A_{\mathfrak{p}} \text{ satisfies } \PP \}, 
\end{align*}
and we call $\PP(A)$ the $\PP$-locus of $A$. It is important to know whether $\PP(A)$ is open in $\Spec(A)$ or not. 

The following lemma is useful to know whether a subset of $\Spec(A)$ is open or not.

\begin{lemma}[{\cite[Theorem 24.2]{Mat86}}]\label{lem:nagatatop}
    Let $A$ be a Noetherian ring, and $U \subset \Spec A$ a subset. Then the following two conditions are necessary and sufficient for $U \subset \Spec A$ to be open. 
    \begin{enumerate}
        \item for $P,\ Q \in \Spec A,\ P \in U$ and $P \supset Q \Rightarrow Q \in U$;
        \item if $P \in U$ then $U$ contains a nonempty open subset of $V(P)$. 
    \end{enumerate}
\end{lemma}

There are many singularities in commutative rings. For example, the following definition introduces the main singularities considered in this paper.

\begin{definition}\label{def:main_sing}
    Let $(A, \mathfrak{m}, k)$ be a Noetherian local ring. We say that $A$ is
    \begin{enumerate}
        \item regular ($\mathsf{Reg}$) if $\dim(A) = \dim_k(\mathfrak{m}/\mathfrak{m}^2)$,
        \item a complete intersection ($\mathsf{CI}$) if $\hat{A} = R/(f_1, \ldots, f_c)$ for some regular local ring $R$ and a regular sequence $f_1, \ldots, f_c$ in $R$,
        \item Gorenstein ($\mathsf{Gor}$) if $A$ is Cohen--Macaulay and has a canonical module $\omega_A$ such that $\omega_A \cong A$,
        \item Cohen--Macaulay ($\mathsf{CM}$) if $\depth(A) = \dim(A)$. 
    \end{enumerate}
\end{definition}

For notation conventions, we introduce the following definition.

\begin{definition}\label{def:main_sing_notation}
    Let $(A, \mathfrak{m}, k)$ be a Noetherian local ring. We put
    \begin{align*}
        \embdim(A) \coloneqq \dim_k(\mathfrak{m}/\mathfrak{m}^2)
    \end{align*}
    and call it the \emph{embedding dimension} of $A$. We call $\embcodim(A) \coloneqq \embdim(A) - \dim(A)$ the \emph{codimension} of $A$. 
\end{definition}

In addition to the above singularities, we consider the following singularities in this paper. 

\begin{definition}\label{def:cicodim}
    Let $(A, \mathfrak{m}, k)$ be a Noetherian local ring and $c \geq 0$ an integer. We say that $A$ is a \emph{complete intersection of codimension at most} $c$ ($\mathsf{CI}_{\leq c}$) if $A$ is a complete intersection and $\embcodim(A) \leq c$. In particular, we say that $A$ is a hypersurface ($\mathsf{HS}$) if $A$ is $\mathsf{CI}_{\leq 1}$.
\end{definition}

Clearly, $\mathsf{CI}_{\leq 0}$ means $\mathsf{Reg}$. All singularities in Definition \ref{def:main_sing} are stable under generalization, i.e., for any Noetherian local ring $(A, \mathfrak{m}, k)$ and any prime ideal $\mathfrak{p}$ of $A$, if $A_{\mathfrak{m}}$ satisfies $\PP$, then $A_{\mathfrak{p}}$ satisfies $\PP$, where $\PP = \mathsf{Reg},\ \mathsf{CI},\ \mathsf{Gor},\ \mathsf{CM}$. If we consider the class of Noetherian local rings that are complete intersections of codimension exactly $c$, then it is not stable under generalization for $c > 0$. For example, let $A = k[[x, y]]/(y^2-x^3)$ and $\mathfrak{p} = (0)$. $A$ is a complete intersection of codimension $1$, but $A_{\mathfrak{p}}$ is a field, which is a complete intersection of codimension $0$. Hence, we consider the class as in Definition \ref{def:cicodim}. Indeed, the class $\mathsf{CI}_{\leq c}$ is stable under generalization; see \cite{Put21}. Moreover, the following hierarchy is well-known:
\begin{align*}
    \mathsf{Reg} \Longrightarrow \mathsf{HS} \Longrightarrow \mathsf{CI}_{\leq 2} \Longrightarrow \cdots \Longrightarrow \mathsf{CI}_{\leq c} \Longrightarrow \cdots \Longrightarrow \mathsf{CI} \Longrightarrow \mathsf{Gor} \Longrightarrow \mathsf{CM}.
\end{align*}

It is well-known that the definition of complete intersection in Definition \ref{def:main_sing} (2) is independent of the choice of a regular local ring $R$ and a regular sequence $f_1, \ldots, f_c$ in $R$. We also note that $c = \embdim(A) - \dim(A)$ for a minimal Cohen presentation $\hat{A} = R/(f_1, \ldots, f_c)$; see \cite[\S 21]{Mat86}.

\section{On the conditions for Noetherian rings}\label{sec:conditions}

First, we define the following condition for a property $\PP$ of local rings for convenience of notation. 

\begin{definition}\label{def:openness}
    Let $\PP$ be a property of local rings and $A$ a Noetherian ring. 
    \begin{enumerate}
        \item We say $A$ satisfies $\PP$ if $A_{\mathfrak{p}}$ satisfies $\PP$ for all $\mathfrak{p} \in \Spec(A)$.
        \item We say $A$ satisfies $\PP$-0 if $\PP(A)$ contains a nonempty open subset of $\Spec(A)$,
        \item We say $A$ satisfies $\PP$-1 if $\PP(A)$ is an open subset of $\Spec(A)$,
        \item We say $A$ satisfies $\PP$-2 if $\PP(B)$ is an open subset of $\Spec(B)$ for any finite type $A$-algebra $B$,
        \item We say $A$ satisfies $\PP$-Q0 if $\PP(A/\mathfrak{p})$ contains a nonempty open subset of $\Spec(A/\mathfrak{p})$ for any $\mathfrak{p} \in \Spec(A)$,
        \item We say $A$ satisfies $\PP$-Q1 if $\PP(A/\mathfrak{p})$ is an open subset of $\Spec(A/\mathfrak{p})$ for any $\mathfrak{p} \in \Spec(A)$,
        \item We say $A$ satisfies $\PP$-Q2 if $\PP(B/\mathfrak{q})$ is an open subset of $\Spec(B/\mathfrak{q})$ for any finite type $A$-algebra $B$ and any $\mathfrak{q} \in \Spec(B)$.
    \end{enumerate}
    ``Q'' in the above conditions stands for ``Quotient''.
\end{definition}

Although we define $\PP$-Q2 because of its theoretical interest, we hardly ever use it in this paper. Moreover, we often assume that fields satisfy $\PP$ in this paper. 

\begin{remark}\label{rem:openness}
    We remark the following two points:
    \begin{enumerate}
        \item Let $\PP$ be a property of local rings. Then, we have the following implications: 
        \begin{align*}
            &\PP\text{-2} \implies \PP\text{-1} \\
            &\PP\text{-2} \implies \PP\text{-Q2} \implies \PP\text{-Q1}.
        \end{align*}
        Moreover, if fields satisfy $\PP$, then we have the following implications:
        \begin{align*}
            \PP\text{-Q1} \implies \PP\text{-Q0}.
        \end{align*}
        In this paper, we consider the case when $\PP$ is a property of local rings such that fields satisfy $\PP$. In this case, we have the following implications:
        \begin{align*}
            \PP\text{-2} \implies \PP\text{-Q2} \implies \PP\text{-Q1} \implies \PP\text{-Q0}.
        \end{align*}
        \item For two properties $\PP_1$ and $\PP_2$ of local rings, if $\PP_1 \Longrightarrow \PP_2$ holds, then we can check easily that $\PP_1$-Q0 $\Longrightarrow \PP_2$-Q0.
    \end{enumerate}
\end{remark}

Whether converses of the above implications hold or not depends on the property $\PP$.

The weakest condition, $\PP$-Q0, is preserved under several operations. 

\begin{proposition}\label{prop:Q0preserved}
    Let $\PP$ be a property of local rings and $A$ be a Noetherian ring. Then, the following hold:
    \begin{enumerate}
        \item If $A$ satisfies $\PP$\textnormal{-Q0}, then $A_a$ satisfies $\PP$\textnormal{-Q0} for any $a \in A \setminus \sqrt{0}$.
        \item If $A$ satisfies $\PP$\textnormal{-Q0}, then $A/I$ satisfies $\PP$\textnormal{-Q0} for any ideal $I \neq A$ of $A$.
        \item If $A$ satisfies $\PP$\textnormal{-Q0}, then $A_{\mathfrak{p}}$ satisfies $\PP$\textnormal{-Q0} for any $\mathfrak{p} \in \Spec(A)$.
        \item If $A$ satisfies $\PP$\textnormal{-Q0}, then $A[x]/(x^r)$ satisfies $\PP$\textnormal{-Q0} for any integer $r \geq 1$.
    \end{enumerate}
\end{proposition}
\begin{proof}
    (1) Let $\mathfrak{p}A_a$ be a prime ideal of $A_a$. Since $A$ satisfies $\PP$-Q0, we can take an open subset $U$ of $\Spec(A)$ such that $\emptyset \neq U \cap V(\mathfrak{p}) \subset \PP(A/\mathfrak{p})$. Taking $\cap D(a)$, we have that $A_a$ satisfies $\PP$-Q0. 

    (2) Let $\mathfrak{p}/I$ be a prime ideal of $A/I$. Since $A$ satisfies $\PP$-Q0, we can take an open subset $U$ of $\Spec(A)$ such that $\emptyset \neq U \cap V(\mathfrak{p}) \subset \PP(A/\mathfrak{p})$. Hence, we have $\emptyset \neq U \cap \Spec(A/I) \cap V(\mathfrak{p}/I) \subset \PP((A/I)/(\mathfrak{p}/I))$, which means $A/I$ satisfies $\PP$-Q0.

    (3) Let $\mathfrak{q}A_{\mathfrak{p}}$ be a prime ideal of $A_{\mathfrak{p}}$. Since $A$ satisfies $\PP$-Q0, we can take an open subset $U$ of $\Spec(A)$ such that $\emptyset \neq U \cap \Spec(A/\mathfrak{q}) \subset \PP(A/\mathfrak{q})$. Then, we have
    \begin{align*}
        \emptyset &\neq U \cap \Spec(A_{\mathfrak{p}}/\mathfrak{q}A_{\mathfrak{p}}) \\ 
        &= U \cap \Spec(A/\mathfrak{q}) \cap \Spec(A_{\mathfrak{p}}) \\
        &\subset \PP(A/\mathfrak{q}) \cap \Spec(A_{\mathfrak{p}}) \\
        &= \PP(A_{\mathfrak{p}}/\mathfrak{q}A_{\mathfrak{p}}). 
    \end{align*}
    Therefore, $A_{\mathfrak{p}}$ satisfies $\PP$-Q0.

    (4) For any $\mathfrak{P} \in \Spec(A[x]/(x^r))$, we have $\mathfrak{P} = (\mathfrak{p}, x)$, where $\mathfrak{p} = \mathfrak{P} \cap A$. Moreover, $(A[x]/(x^r))/\mathfrak{P} \cong A/\mathfrak{p}$. Thus, (4) is immediate. 
\end{proof}

These properties (1) to (3) are used in the proof of Proposition \ref{prop:conditions-imply-NC} and (4) is used in the proof of Theorem \ref{th:counterexampleNC}.

Nagata considered the following condition for a property $\PP$ of local rings in \cite{Nag59}, called the Nagata criterion, \NC{}.

\begin{definition}[Nagata criterion]\label{def:Nagata_criterion}
    Let $\PP$ be a property of local rings. We say that $\PP$ satisfies the \emph{Nagata criterion} if for any Noetherian ring $A$, $\PP$-Q0 implies $\PP$-1. In other words, we say that $\PP$ satisfies the Nagata criterion if for any Noetherian ring $A$, the following condition holds:
    \begin{enumerate}
        \item[(NC)] If $\PP(A/\mathfrak{p})$ contains a nonempty open subset of $\Spec(A/\mathfrak{p})$ for any $\mathfrak{p} \in \Spec(A)$, then $\PP(A)$ is an open subset of $\Spec(A)$.
    \end{enumerate}
\end{definition}

\begin{remark}\label{rem:NC}
    Whether \NC{} holds or not depends on the property $\PP$. Note that, for properties $\PP_1,\ \PP_2$ of local rings with $\PP_1 \Longrightarrow \PP_2$, the satisfaction of \NC{} for $\PP_1$ does not imply its satisfaction for $\PP_2$. For example, the following singularities satisfy \NC{}. 
    \begin{enumerate}
        \item $\PP = \mathsf{Reg}$      (\cite{Nag59}, \cite[Theorem 24.4]{Mat86})
        \item $\PP = \mathsf{CI}$       (\cite[Theorem 3.1]{GM78})
        \item $\PP = \mathsf{Gor}$      (\cite[Theorem 1.4]{GM78}, \cite[Theorem 24.6]{Mat86})
        \item $\PP = \mathsf{CM}$       (\cite{MV77}, \cite[Theorem 24.5]{Mat86})
        \item $\PP = (\mathsf{T}_n)$    (\cite[Theorem 2.3]{GM78})
        \item $\PP = (\mathsf{S}_n)$    (\cite{MV77}, \cite[Theorem 2.2]{Tak99})
        \item $\PP = (\mathsf{R}_n)$    (\cite{MV77}, \cite[Theorem 3.2]{Tak99})
        \item $\PP = \mathsf{H}_n$      (\cite[Proposition 3.1]{Rag80})
    \end{enumerate}
    However, \NC{} does not hold for $\PP = \mathsf{Buchsbaum}$ (\cite[Theorem]{Kan91}). 
\end{remark}

Moreover, Valabrega considered the following condition for a property $\PP$ of local rings in \cite[Definition 3]{Val78}, which he called the quotient condition, \QC{}. 

\begin{definition}[Quotient condition]\label{def:quotient_condition}
    Let $\PP$ be a property of local rings. We say that $\PP$ satisfies the \emph{quotient condition} if for any Noetherian ring $A$, $\PP$ implies $\PP$-Q0. In other words, we say that $\PP$ satisfies the quotient condition if for any Noetherian ring $A$, the following condition holds:
    \begin{enumerate}
        \item[(QC)] If $A_{\mathfrak{p}}$ satisfies $\PP$ for all $\mathfrak{p} \in \Spec(A)$, then $\PP(A/\mathfrak{p})$ contains a nonempty open subset of $\Spec(A/\mathfrak{p})$ for any $\mathfrak{p} \in \Spec(A)$.
    \end{enumerate}
\end{definition}

\QC{} is a converse-like condition to \NC{}. 

If we find that a property $\PP$ of local rings satisfies \NC{}, it is also worth considering whether \QC{} holds or not because of the following proposition.

\begin{proposition}\label{prop:NCandQC}
    Let $\PP$ be a property of local rings. Assume that $\PP$ satisfies \NC{} and \QC{}. Let $A \twoheadrightarrow B$ be a surjective ring homomorphism. Then, if $A$ satisfies $\PP$, $\PP(B)$ is an open subset of $\Spec(B)$.
\end{proposition}
\begin{proof}
    We have $B = A/I$ for some ideal $I$ of $A$. For any prime ideal $\mathfrak{q} \in \Spec(B)$, we have $B/\mathfrak{q} = A/\mathfrak{p}$, where $\mathfrak{p} = A \cap \mathfrak{q}$. Since $\PP$ satisfies \QC{} and $A$ satisfies $\PP$, $\PP(A/\mathfrak{p})$ contains a nonempty open subset of $\Spec(A/\mathfrak{p})$. In other words, $\PP(B/\mathfrak{q})$ contains a nonempty open subset of $\Spec(B/\mathfrak{q})$. Since $\PP$ satisfies \NC{}, $\PP(B)$ is an open subset of $\Spec(B)$.
\end{proof}

We often consider the case when regular implies $\PP$ for a property $\PP$ of local rings. In this case, by Cohen's structure theorem, for any complete Noetherian local ring $B$, we can find a surjective ring homomorphism $A \twoheadrightarrow B$ from a regular, in particular $\PP$, local ring $A$. Hence, if $\PP$ satisfies \NC{} and \QC{}, then $\PP(B)$ is an open subset of $\Spec(B)$ for any complete Noetherian local ring $B$ by Proposition \ref{prop:NCandQC}.

It is well known that the regular, complete intersection, Gorenstein, and Cohen--Macaulay properties satisfy \NC{} and are stable under polynomial extensions. However, we make the following remarks about \QC{}. 

\begin{remark}\label{rem:regular_QC}
    The following singularities satisfy \QC{}:
    \begin{enumerate}
        \item $\PP = \mathsf{CI}$       (\cite[Corollary 3.4]{GM78})
        \item $\PP = \mathsf{Gor}$      (\cite[Corollary 1.6]{GM78}, \cite[Exercise 24.3]{Mat86})
        \item $\PP = \mathsf{CM}$       (\cite[Theorem 2.1]{MV77}, \cite[Exercise 24.2]{Mat86})
    \end{enumerate}
    However, \QC{} does not hold for the regular property in general, as the following example shows. 

    Let $\PP = \mathsf{Reg}$ be the property of being regular. Then, $\mathsf{Reg}$ does not satisfy \QC{} in general. In \cite[{\S 4}]{Nag59}, Nagata gave an example of a regular local ring $A$ such that $A$ does not satisfy $\mathsf{Reg}$-2. In other words, he constructed a regular local ring $A$ which is not J-2. More precisely, we can find a regular local ring $A$ and an element $c$ in a domain containing $A$ such that $\mathsf{Reg}(A[c])$ does not contain any nonempty open subset of $\Spec(A[c])$. Hence, a polynomial extension of $A$ in one variable $A[x]$ does not satisfy $\mathsf{Reg}$-Q0. This is because we can take a prime ideal $\mathfrak{p}$ of $A[x]$ such that $A[x]/\mathfrak{p} \cong A[c]$ and $\mathsf{Reg}(A[x]/\mathfrak{p})$ does not contain any nonempty open subset of $\Spec(A[x]/\mathfrak{p})$ although $A[x]$ satisfies $\mathsf{Reg}$. 
\end{remark}

Another reason to consider \NC{} is that we can claim that $\PP(A)$ itself is open if $\PP$ satisfies \NC{} and $A$ satisfies $\mathsf{Reg}$-Q0. $\mathsf{Reg}$-Q0 means that the regular locus $\mathsf{Reg}(A/\mathfrak{p})$ contains a nonempty open subset of $\Spec(A/\mathfrak{p})$ for every prime ideal $\mathfrak{p}$ of $A$. More precisely, we have the following proposition.

\begin{proposition}\label{prop:NCandExcellent}
    Let $\PP$ be a property of local rings. Assume that regular implies $\PP$ and $\PP$ satisfies \NC{}. Let $A$ be a Noetherian ring which satisfies $\mathsf{Reg}$\textnormal{-Q0}. Then, $\PP(A)$ is an open subset of $\Spec(A)$.
\end{proposition}
\begin{proof}
    Let $\mathsf{Reg}$ be the property of being a regular local ring.

    Let $\mathfrak{p} \in \Spec(A)$. Since $A$ is $\mathsf{Reg}$-Q0, $\mathsf{Reg}(A/\mathfrak{p})$ contains a nonempty open subset of $\Spec(A/\mathfrak{p})$. Since regular implies $\PP$, $\PP(A/\mathfrak{p})$ contains a nonempty open subset of $\Spec(A/\mathfrak{p})$. Since $\mathfrak{p} \in \Spec(A)$ is arbitrary, $A$ satisfies $\PP$-Q0. Since $\PP$ satisfies \NC{}, $\PP(A)$ is an open subset of $\Spec(A)$.
\end{proof}

In particular, $\PP(A)$ is an open subset of $\Spec(A)$ for any excellent ring $A$ if regularity implies $\PP$ and $\PP$ satisfies \NC{}. 

Almost all commutative rings in algebraic geometry are excellent rings. Hence, we can say that $\PP(A)$ is an open subset of $\Spec(A)$ for almost all commutative rings $A$ in algebraic geometry in terms of Proposition \ref{prop:NCandExcellent} if regularity implies $\PP$ and $\PP$ satisfies \NC{}. Even if we do not have that $\PP$ satisfies \NC{}, it is still worth considering whether $\PP(A)$ is an open subset of $\Spec(A)$ for an excellent ring $A$ or not. 

The Nagata criterion is the implication $\PP$-Q0 $\Longrightarrow$ $\PP$-1. Therefore, let us consider whether the other implications hold or not. We can claim the converse of Remark \ref{rem:openness} (1) under the assumption that $\PP$ satisfies some additional conditions including \NC{} in the following proposition. 

\begin{proposition}\label{prop:opennessconv}
    Let $\PP$ be a property of local rings. Then, we have the following implications:
    \begin{enumerate}
        \item If $\PP$ satisfies \NC{}, then $\PP\textnormal{-Q0} \implies \PP\textnormal{-Q1}$ and $\PP\textnormal{-Q2} \implies \PP\textnormal{-2}$.
        \item Assume that fields satisfy $\PP$. If $\PP$ satisfies \NC{}, \QC{} and stability under taking polynomial extensions, then $\PP\textnormal{-Q1} \implies \PP\textnormal{-Q2}$.
    \end{enumerate}
\end{proposition}
\begin{proof}
    (1) First, we prove that $\PP\textnormal{-Q0} \implies \PP\textnormal{-Q1}$ under the assumption that $\PP$ satisfies \NC{}. Let $A$ be a Noetherian ring that satisfies $\PP\textnormal{-Q0}$. Fix a prime ideal $\mathfrak{p} \in \Spec(A)$ arbitrarily. We have to show that $\PP(A/\mathfrak{p})$ is an open subset of $\Spec(A/\mathfrak{p})$. For any prime ideal $\mathfrak{q} \in \Spec(A/\mathfrak{p}) = V(\mathfrak{p})$, we have:
    \begin{align*}
        \PP(A/\mathfrak{q}) = \PP((A/\mathfrak{p})/(\mathfrak{q}/\mathfrak{p})).
    \end{align*}
    Since $A$ satisfies $\PP\textnormal{-Q0}$, $\PP(A/\mathfrak{q})$ contains a nonempty open subset of $\Spec(A/\mathfrak{q})$. In other words, $\PP((A/\mathfrak{p})/(\mathfrak{q}/\mathfrak{p}))$ contains a nonempty open subset of $\Spec((A/\mathfrak{p})/(\mathfrak{q}/\mathfrak{p}))$. Since $\PP$ satisfies \NC{}, $\PP(A/\mathfrak{p})$ is an open subset of $\Spec(A/\mathfrak{p})$. Since $\mathfrak{p} \in \Spec(A)$ is arbitrary, we conclude that $A$ satisfies $\PP\textnormal{-Q1}$. Second, we prove that $\PP\textnormal{-Q2} \implies \PP\textnormal{-2}$ under the assumption that $\PP$ satisfies \NC{}. Let $A$ be a Noetherian ring that satisfies $\PP\textnormal{-Q2}$. Fix a finite type $A$-algebra $B$ arbitrarily. We have to show that $\PP(B)$ is an open subset of $\Spec(B)$. For any prime ideal $\mathfrak{q} \in \Spec(B)$, $\PP(B/\mathfrak{q})$ is an open subset of $\Spec(B/\mathfrak{q})$ since $A$ satisfies $\PP\textnormal{-Q2}$. Since $\PP$ satisfies \NC{}, $\PP(B)$ is an open subset of $\Spec(B)$.

    (2) We prove that $\PP\textnormal{-Q1} \implies \PP\textnormal{-Q2}$ under the assumption that $\PP$ satisfies \NC{}, \QC{} and stability under taking polynomial extensions. Since fields satisfy $\PP$, $\PP$-Q1 implies $\PP$-Q0 by Remark \ref{rem:openness}. Thus, it is enough to prove that $\PP$-Q0 implies $\PP$-Q2. Let $A$ be a Noetherian ring that satisfies $\PP\textnormal{-Q0}$. Fix a finite type $A$-algebra $B$ arbitrarily. We have to show that $\PP(B/\mathfrak{q})$ is an open subset of $\Spec(B/\mathfrak{q})$ for any $\mathfrak{q} \in \Spec(B)$. Consider the natural mapping:
    \begin{align*}
        A \longrightarrow A[x_1,\ldots, x_n]/I = B \longrightarrow B/\mathfrak{q}.
    \end{align*}
    Let $\mathfrak{p}$ be the kernel of the composition $A \longrightarrow B/\mathfrak{q}$. We may regard $A/\mathfrak{p}$ as a subring of $B/\mathfrak{q}$. Since $B/\mathfrak{q}$ is a domain, $\mathfrak{p}$ is a prime ideal of $A$. Moreover, $B/\mathfrak{q}$ is a finite type $A/\mathfrak{p}$-algebra. Since $A$ satisfies $\PP\textnormal{-Q0}$, $\PP(A/\mathfrak{p})$ contains a nonempty open subset of $\Spec(A/\mathfrak{p})$. Hence, we can take an element $f \in A \setminus \mathfrak{p}$ such that $A_f/\mathfrak{p}A_f$ satisfies $\PP$. Then, we have:
    \begin{align*}
        (B/\mathfrak{q})_f &= (A[x_1,\ldots, x_n]/I)_f \otimes_{A_f} A_f/\mathfrak{p}A_f \\
        &= (A_f/\mathfrak{p}A_f)[x_1,\ldots, x_n]/(I_f). 
    \end{align*}
    By the stability of $\PP$ under taking polynomial extensions, $(A_f/\mathfrak{p}A_f)[x_1,\ldots, x_n]$ satisfies $\PP$. Since $\PP$ satisfies \NC{} and \QC{}, $\PP((B/\mathfrak{q})_f)$ is an open subset of $\Spec((B/\mathfrak{q})_f)$ by Proposition \ref{prop:NCandQC}. Hence, $\PP(B/\mathfrak{q})$ contains a nonempty open subset of $\Spec(B/\mathfrak{q})$ because we have:
    \begin{align*}
        \emptyset \neq \PP((B/\mathfrak{q})_f) = \PP(B/\mathfrak{q}) \cap D_{B/\mathfrak{q}}(f) \subset \PP(B/\mathfrak{q}) \subset \Spec(B/\mathfrak{q}).
    \end{align*}
    Since $\mathfrak{q} \in \Spec(B)$ is arbitrary, \NC{} shows that $\PP(B/\mathfrak{q})$ is open in $\Spec(B/\mathfrak{q})$ for every $\mathfrak{q} \in \Spec(B)$. Hence, $A$ satisfies $\PP\textnormal{-Q2}$.
\end{proof}

The summary of implications is as follows. Dashed arrows mean that these implications hold under the assumptions of the labeled numbers below. 

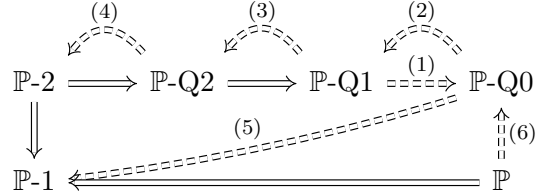
\[\begin{tikzcd}
	{\mathbb{P}\text{-2}} & {\mathbb{P}\text{-Q2}} & {\mathbb{P}\text{-Q1}} & {\mathbb{P}\text{-Q0}} \\
	{\mathbb{P}\text{-1}} &&& {\mathbb{P}}
	\arrow[Rightarrow, from=1-1, to=1-2]
	\arrow[Rightarrow, from=1-1, to=2-1]
	\arrow["{(4)}"', shift right, curve={height=12pt}, Rightarrow, dashed, from=1-2, to=1-1]
	\arrow[Rightarrow, from=1-2, to=1-3]
	\arrow["{(3)}"', shift right, curve={height=12pt}, Rightarrow, dashed, from=1-3, to=1-2]
	\arrow["{(1)}", Rightarrow, dashed, from=1-3, to=1-4]
	\arrow["{(2)}"', shift right, curve={height=12pt}, Rightarrow, dashed, from=1-4, to=1-3]
	\arrow["{(5)}"', curve={height=-6pt}, Rightarrow, dashed, from=1-4, to=2-1]
	\arrow["{(6)}"', Rightarrow, dashed, from=2-4, to=1-4]
	\arrow[Rightarrow, from=2-4, to=2-1]
\end{tikzcd}\]

\begin{enumerate}
    \item Fields satisfy $\PP$. 
    \item $\PP$ satisfies \NC{}.
    \item Fields satisfy $\PP$ and $\PP$ satisfies \NC{}, \QC{} and stability under taking polynomial extensions.
    \item $\PP$ satisfies \NC{}.
    \item $\PP$ satisfies \NC{}. This implication is just the definition of \NC{}.
    \item $\PP$ satisfies \QC{}. This implication is just the definition of \QC{}.
\end{enumerate}

\section{Conditions for the Nagata criterion}\label{sec:conditions-for-NC}

In this section, we give a necessary and sufficient condition for a property $\PP$ of local rings to satisfy \NC{}. First, we define the following conditions for a property $\PP$ of Noetherian local rings.

\begin{definition}\label{def:conditions-for-NC}
    We define the following conditions for a property $\PP$ of
    Noetherian local rings:
    \begin{enumerate}
        \item[(A-1)] $\PP$ is stable under localization, i.e., for any Noetherian ring $A$ and any prime ideals $\mathfrak{q} \subseteq \mathfrak{p}$ of $A$, if $A_{\mathfrak{p}}$ satisfies $\PP$, then $A_{\mathfrak{q}}$ satisfies $\PP$. 
        \item[(A-2)] For any Noetherian ring $A$ and any $P \in \PP(A)$, set $n=\height(P)$. Then there exist elements $x_1,\ldots,x_n \in P$ such that, setting $I=(x_1,\ldots,x_n)$, the following conditions hold:
        \begin{enumerate}
            \item[{(A-2-\makebox[1.2em][c]{i})}] The sequence $x_1/1,\ldots,x_n/1$ is an $A_P$-regular sequence, and
            \begin{align*}
                \sqrt{IA_P}=PA_P.
            \end{align*}
            \item[{(A-2-\makebox[1.2em][c]{ii})}] The prime ideal $P/I$ belongs to $\PP(A/I)$; equivalently, $A_P/IA_P$ satisfies $\PP$.
            \item[{(A-2-\makebox[1.2em][c]{iii})}] There exists an open neighbourhood $W$ of $P$ in $\Spec(A)$ such that, for any $Q \in W \cap V(P)$, if $A_Q/IA_Q$ satisfies $\PP$, then $A_Q$ satisfies $\PP$.
        \end{enumerate}
        \item[(A-3)] Fix a Noetherian ring $A$ satisfying $\PP$-Q0 and any prime ideal $\mathfrak{p} \in \PP(A)$ such that $\mathfrak{p}^r=0$ for some integer $r \geq 1$. Moreover, we assume $\mathfrak{p}^i/\mathfrak{p}^{i+1}$ is a free $A/\mathfrak{p}$-module with finite rank for all $i = 0,\ldots,r-1$. Then there exists an open neighbourhood $U$ of $\mathfrak{p}$ in $\Spec(A)$ such that, for any $\mathfrak{q} \in U$, if $A_{\mathfrak{q}}/\mathfrak{p}A_{\mathfrak{q}}$ satisfies $\PP$, then $A_{\mathfrak{q}}$ satisfies $\PP$. 
    \end{enumerate}
\end{definition}

Almost all commutative algebraists will think that (A-1) is natural. (A-2) seems to be technical, but this is a natural generalization of the condition that a property $\PP$ of local rings is stable under taking quotient by a regular sequence. If we consider a condition for $\PP$ such that for every Noetherian local ring $A$ and every regular sequence $x_1,\ldots,x_n$ in $A$, 
\begin{align*}
    A \text{ satisfies } \PP \Longleftrightarrow A/(x_1,\ldots,x_n) \text{ satisfies } \PP,
\end{align*}
then we can see that regularity does not satisfy $\Longrightarrow$. Therefore, we need to consider a condition weaker than the one above, such as (A-2). (A-3) is a technical condition. Geometrically, (A-3) is a condition that allows us to lift the property $\PP$ from a closed subscheme defined by a nilpotent ideal to an open neighbourhood of the subscheme.

We note some remarks on the above conditions.

\begin{remark}\label{rem:A2-implies-CM}
    We note the following remarks. 
    \begin{enumerate}
        \item The conditions in Definition \ref{def:conditions-for-NC} are designed for studying \NC{}. Hence, when studying \NC{}, we may assume that the Noetherian rings occurring in (A-1), (A-2), and (A-3) satisfy $\PP$-Q0. 
        \item Condition \textup{(A-2)} implies that every Noetherian local ring satisfying $\PP$ is Cohen--Macaulay. Indeed, let $(A,\mathfrak{m})$ be a Noetherian local ring satisfying $\PP$. Applying \textup{(A-2)} to $\mathfrak{m} \in \PP(A)$, we obtain an $A$-regular sequence of length
            \begin{align*}
                \height(\mathfrak{m})=\dim(A).
            \end{align*}
            It follows that
            \begin{align*}
                \depth(A) \geq \dim(A).
            \end{align*}
            Since the reverse inequality always holds, we have
            \begin{align*}
                \depth(A) = \dim(A).
            \end{align*}
            Therefore, $A$ is Cohen--Macaulay.
        \item In condition (A-3), $A$ is normally flat along $V(\mathfrak{p})$. In other words, 
            \begin{align*}
                \operatorname{gr}_{\mathfrak{p}}(A) = \bigoplus_{i \geq 0} \mathfrak{p}^i/\mathfrak{p}^{i+1} = \left( A/\mathfrak{p} \right) \oplus \left( \mathfrak{p}/\mathfrak{p}^2 \right) \oplus \cdots \oplus \left( \mathfrak{p}^{r-1}/\mathfrak{p}^r \right)
            \end{align*}
            is a flat $A/\mathfrak{p}$-module. Normal flatness seems to have been introduced by Hironaka in \cite[I. Ch. 0. \S 4. Definition 1]{Hir64} to study the resolution of singularities. 
    \end{enumerate}    
\end{remark}

We give a necessary and sufficient condition for a property $\PP$ of Noetherian local rings to satisfy \NC{} using Lemma \ref{lem:nagatatop}. 

\begin{proposition}\label{prop:conditions-imply-NC}
    Let $\PP$ be a property of Noetherian local rings. Then, the following statements hold:
    \begin{enumerate}
        \item If $\PP$ satisfies \NC{}, then $\PP$ satisfies \textup{(A-3)}. 
        \item If $\PP$ satisfies \textup{(A-1)}, \textup{(A-2)}, and \textup{(A-3)}, then $\PP$ satisfies \NC{}.
    \end{enumerate}
\end{proposition}
\begin{proof}
    (1) Let $A$ be a Noetherian ring satisfying $\PP$-Q0 and $\mathfrak{p} \in \PP(A)$ such that $\mathfrak{p}^r=0$ for some integer $r \geq 1$. Moreover, we assume $\mathfrak{p}^i/\mathfrak{p}^{i+1}$ is a free $A/\mathfrak{p}$-module with finite rank for all $i = 0,\ldots,r-1$. We have to show that there exists an open neighbourhood $U$ of $\mathfrak{p}$ in $\Spec(A)$ such that, for any $\mathfrak{q} \in U$, if $A_{\mathfrak{q}}/\mathfrak{p}A_{\mathfrak{q}}$ satisfies $\PP$, then $A_{\mathfrak{q}}$ satisfies $\PP$. Since $A$ satisfies $\PP$-Q0, we have that $\PP(A/\mathfrak{p})$ contains a nonempty open subset of $\Spec(A/\mathfrak{p})$. Hence, there exists an open subset $O$ of $\Spec(A)$ such that
    \begin{align*}
        \emptyset \neq O \cap V(\mathfrak{p}) \subset \PP(A/\mathfrak{p}).
    \end{align*}
    Note that $\mathfrak{p} \in O \cap V(\mathfrak{p})$. Since $\PP$ satisfies \NC{}, $\PP(A)$ is an open subset of $\Spec(A)$. Therefore, we put $U = O \cap \PP(A)$. Then, we find that $\mathfrak{p} \in U \neq \emptyset$. Moreover, for any $\mathfrak{q} \in U \cap V(\mathfrak{p})$, we have $\mathfrak{q} \in \PP(A/\mathfrak{p})$ and $\mathfrak{q} \in \PP(A)$. In particular, the implication $A_{\mathfrak{q}}/\mathfrak{p}A_{\mathfrak{q}} \text{ satisfies } \PP \Longrightarrow A_{\mathfrak{q}} \text{ satisfies } \PP$ holds for any $\mathfrak{q} \in U \cap V(\mathfrak{p})$. Hence, we conclude that $\PP$ satisfies (A-3).

    (2) Let $A$ be a Noetherian ring which satisfies $\PP$-Q0, i.e., $\PP(A/\mathfrak{p})$ contains a nonempty open subset of $\Spec(A/\mathfrak{p})$ for every $\mathfrak{p} \in \Spec(A)$. We show that $\PP(A)$ is open in $\Spec(A)$.

    By (A-1), the subset $\PP(A)$ is stable under localization. Hence, it is enough to show that, for every $P \in \PP(A)$, the subset $\PP(A)$ contains a nonempty open subset of $V(P)$ by Lemma \ref{lem:nagatatop}.

    First, if we show that $\PP(A_a)$ contains a nonempty open subset of $V(PA_a)$ for some $a \in A \setminus P$, then $\PP(A)$ contains a nonempty open subset of $V(P)$. Indeed, if we have it, then there exists an open subset $U$ of $\Spec(A)$ such that
    \begin{align*}
        \emptyset \neq V(PA_a) \cap U \subset \PP(A_a).
    \end{align*}
    This means that
    \begin{align*}
        \emptyset \neq V(P) \cap D(a) \cap U \subset \PP(A_a) = \PP(A) \cap D(a) \subset \PP(A).
    \end{align*}
    Hence, $\PP(A)$ contains a nonempty open subset of $V(P)$. Moreover, we have:
    \begin{align*}
        (A_a)_{PA_a} = A_P.
    \end{align*}
    Thus, we have $PA_a \in \PP(A_a)$. Therefore, we may replace $A$ by $A_a$ for some $a \in A \setminus P$. Notice that $A_a$ also satisfies $\PP$-Q0 by Proposition \ref{prop:Q0preserved}. 

    Fix $P \in \PP(A)$. By (A-2), there exist elements $x_1,\ldots,x_n \in P$, where $n = \operatorname{ht}(P)$, satisfying (A-2-i), (A-2-ii), and (A-2-iii). Put $I = (x_1,\ldots,x_n)$. By (A-2-i), the sequence $x_1/1,\ldots,x_n/1$ is an $A_P$-regular sequence. Hence, there exists $a_1 \in A \setminus P$ such that the images of $x_1,\ldots,x_n$ form an $A_{a_1}$-regular sequence. Moreover, (A-2-i) gives $\sqrt{IA_P} = PA_P$. Hence, there exists $a_2 \in A \setminus P$ such that $\sqrt{IA_{a_2}} = PA_{a_2}$. By (A-2-iii), there exists an open neighbourhood $W$ of $P$ in $\Spec(A)$ such that, for every $Q \in W \cap V(P)$, if $A_Q/IA_Q$ satisfies $\PP$, then $A_Q$ satisfies $\PP$. Hence, there exists $a_3 \in A \setminus P$ such that, for every $Q \in D(a_3) \cap V(P) \subset W \cap V(P)$, if $A_Q/IA_Q$ satisfies $\PP$, then $A_Q$ satisfies $\PP$. Replacing $A$ by $A_a$, where $a = a_1 a_2 a_3$, we may assume that the following conditions hold:
    \begin{enumerate}
        \item[(2-1)] $x_1,\ldots,x_n$ form an $A$-regular sequence and $\sqrt{IA} = P$,
        \item[(2-2)] $P/I \in \PP(A/I)$; equivalently, $A_P/IA_P$ satisfies $\PP$,
        \item[(2-3)] for every $Q \in V(P)$, if $A_Q/IA_Q$ satisfies
        $\PP$, then $A_Q$ satisfies $\PP$.
    \end{enumerate}
    
    If we show that $\PP(A/I)$ contains a nonempty open subset of $V(P/I)$, then $\PP(A)$ contains a nonempty open subset of $V(P)$. Indeed, if we have it, then there exists an open subset $O$ of $\Spec(A/I)$ such that
    \begin{align*}
        \emptyset \neq O \cap V(P/I) \subset \PP(A/I). 
    \end{align*}
    This means that
    \begin{align*}
        \emptyset \neq O \cap V(P) \subset \PP(A/I) \cap V(P) \subset \PP(A).
    \end{align*}
    Note that the last inclusion follows from (2-3). Thus, we may replace $A$ by $A/I$. Notice that $A/I$ also satisfies $\PP$-Q0 by Proposition \ref{prop:Q0preserved}. In this case, we have $P^r = 0$ for some integer $r \geq 1$, $P \in \PP(A)$. Moreover, by generic freeness, we can assume $P^i/P^{i+1}$ is a free $A/P$-module with finite rank for all $i = 0,\ldots,r-1$ by taking some neighborhood of $P$. Thus, by (A-3), there exists an open neighbourhood $U$ of $P$ in $\Spec(A)$ such that, for every $Q \in U$, if $A_Q/PA_Q$ satisfies $\PP$, then $A_Q$ satisfies $\PP$. Since $A$ satisfies $\PP$-Q0, $\PP(A/P)$ contains a nonempty open subset of $\Spec(A/P)$. Hence, there exists an open subset $O$ of $\Spec(A)$ such that
    \begin{align*}
        \emptyset \neq O \cap V(P) \subset \PP(A/P).
    \end{align*}
    Hence, we have
    \begin{align*}
        \emptyset \neq U \cap O \cap V(P) \subset \PP(A/P) \cap U \subset \PP(A). 
    \end{align*}
    Note that the last inclusion follows from (A-3). Therefore, $\PP(A)$ contains a nonempty open subset of $V(P)$. Hence, $\PP(A)$ is open in $\Spec(A)$.
\end{proof}

Let $\PP$ be a property of Noetherian local rings. Whether $\PP$ satisfies \NC{} or not depends on the property $\PP$. Hence, the proof of \NC{} is different for each property $\PP$. In this section, we prove \NC{} for $\mathsf{Reg},\ \mathsf{CI},\ \mathsf{Gor},\ \mathsf{CM}$ simultaneously using Proposition \ref{prop:conditions-imply-NC}. However, it is easy to show that all four properties satisfy (A-1) and (A-2). Hence, we check that they satisfy (A-3) in Definition \ref{def:conditions-for-NC} to show that they satisfy \NC{}. Indeed, this verification follows almost directly from earlier work. We summarize the result in the following proposition.

\begin{proposition}\label{prop:known_singularity_NC}
    The properties $\mathsf{Reg},\ \mathsf{CI},\ \mathsf{Gor},\ \mathsf{CM}$ of Noetherian local rings satisfy \textup{(A-3)}.
\end{proposition}
\begin{proof}
    ($\mathsf{Reg}$): Let $A$ be a Noetherian ring and $\mathfrak{p} \in \mathsf{Reg}(A)$ with $\mathfrak{p}^r = 0$ for some $r > 0$. Moreover, we assume $\mathfrak{p}/\mathfrak{p}^2$ is a free $A/\mathfrak{p}$-module with finite rank. Then, we have:
    \begin{align*}
        \mathfrak{p}/\mathfrak{p}^2 = (A/\mathfrak{p})^{\oplus e}
    \end{align*}
    for some integer $e \geq 0$. Since $A_{\mathfrak{p}}$ is zero-dimensional and regular, this ring is a field. Hence, tensoring the above equality with $A_{\mathfrak{p}}$ over $A$ gives $e=0$. Therefore, $\mathfrak{p} = 0$. We can easily see that (A-3) holds in this case.

    ($\mathsf{CI}$): We use Andr\'{e}--Quillen homology. Let $A$ be a Noetherian ring and $\mathfrak{p} \in \mathsf{CI}(A)$ with $\mathfrak{p}^r = 0$ for some $r > 0$. Moreover, we assume $\mathfrak{p}^i/\mathfrak{p}^{i+1}$ is a free $A/\mathfrak{p}$-module with finite rank for all $i = 0,\ldots,r-1$. We put
    \begin{align*}
        H_i = H_i(A, A/\mathfrak{p}, A/\mathfrak{p}) \quad (i = 0, 1, 2, 3, 4). 
    \end{align*}
    Since $A_{\mathfrak{p}}$ is a complete intersection, we have
    \begin{align*}
        H_i \otimes_A A_{\mathfrak{p}} = H_i(A_{\mathfrak{p}}, A_{\mathfrak{p}}/\mathfrak{p}A_{\mathfrak{p}}, A_{\mathfrak{p}}/\mathfrak{p}A_{\mathfrak{p}}) = 0 \quad (i = 3, 4).
    \end{align*}
    Moreover, $H_i$ are finitely generated $A/\mathfrak{p}$-modules. Hence, we can take an element $a \in A \setminus \mathfrak{p}$ such that
    \begin{align*}
        H_i \otimes_{A/\mathfrak{p}} (A/\mathfrak{p})_a \text{ are free } (A/\mathfrak{p})_a \text{-modules } (i=0,1,2) \text{ and } H_i \otimes_A A_a = 0 \quad (i = 3, 4).
    \end{align*}
    Let $U = D(a)$. Then, for any $\mathfrak{q} \in U$, we have
    \begin{align*}
        \begin{cases}
            H_i(A_{\mathfrak{q}}, A_{\mathfrak{q}}/\mathfrak{p}A_{\mathfrak{q}}, A_{\mathfrak{q}}/\mathfrak{p}A_{\mathfrak{q}}) \text{ is a free } A_{\mathfrak{q}}/\mathfrak{p}A_{\mathfrak{q}} \text{-module } (i=0,1,2), \\
            H_i(A_{\mathfrak{q}}, A_{\mathfrak{q}}/\mathfrak{p}A_{\mathfrak{q}}, A_{\mathfrak{q}}/\mathfrak{p}A_{\mathfrak{q}}) = 0 \quad (i = 3, 4).
        \end{cases}
    \end{align*}
    Hence, by \cite[Lemma 3.2]{GM78}, we conclude that $A_{\mathfrak{q}}/\mathfrak{p}A_{\mathfrak{q}}$ is a complete intersection if and only if $A_{\mathfrak{q}}$ is a complete intersection. Therefore, we conclude that $\mathsf{CI}$ satisfies (A-3).

    ($\mathsf{Gor}$): See the proof of \cite[Theorem 24.6]{Mat86}.

    ($\mathsf{CM}$): See the proof of \cite[Theorem 24.5]{Mat86}.
\end{proof}

\section{Complete intersections of bounded codimension}\label{sec:ci-bounded}

In this section, we consider the property $\PP = \mathsf{CI}_{\leq c}$ of Noetherian local rings (Definition \ref{def:cicodim}). This class of singularities stands between the property $\mathsf{Reg}$ and $\mathsf{CI}$. Hence, we can expect that $\mathsf{CI}_{\leq c}$ also satisfies good properties like $\mathsf{Reg}$ and $\mathsf{CI}$. In fact, we can show various properties of $\mathsf{CI}_{\leq c}$ in this section. However, we will see that $\mathsf{CI}_{\leq c}$ does not satisfy \NC{} in general. Therefore, as we see in the previous section, \QC{} is less useful for $\mathsf{CI}_{\leq c}$, since Proposition \ref{prop:NCandQC} cannot be applied to prove openness of its locus. However, we can show that $\mathsf{CI}_{\leq c}(A)$ is an open subset of $\Spec(A)$ for every $\mathsf{Reg}$-Q0 ring $A$, just as for the properties $\mathsf{Reg},\ \mathsf{CI},\ \mathsf{Gor},\ \mathsf{CM}$. 

\subsection{Complete intersections of bounded codimension}\label{subsec:ci-bounded}

We study basic properties of complete intersections of bounded codimension. Let $(A, \mathfrak{m})$ be a Noetherian local ring. For any regular sequence $x_1,\ldots,x_n$ in $A$, the following is well-known.
\begin{align*}
    A \text{ is a complete intersection } \Longleftrightarrow A/(x_1,\ldots,x_n) \text{ is a complete intersection}.
\end{align*}
However, the analogous equivalence does not hold for complete intersections of bounded codimension. In fact, we have the following proposition.

\begin{proposition}\label{prop:cibounded_regseq}
    Let $(A, \mathfrak{m}, k)$ be a Noetherian local ring and $x_1,\ldots,x_n$ be a regular sequence in $A$. Let $\overline{x_1},\ldots,\overline{x_n}$ be the images of $x_1,\ldots,x_n$ in $\mathfrak{m}/\mathfrak{m}^2$. We put 
    \begin{align*}
        s = \dim_k(\operatorname{span}_k\{\overline{x_1},\ldots,\overline{x_n}\}) \leq n. 
    \end{align*}
    Then, we have
    \begin{align*}
        \embdim(A) - \dim(A) = \embdim(A/(x_1,\ldots,x_n)) - \dim(A/(x_1,\ldots,x_n)) + s - n.
    \end{align*}
\end{proposition}
\begin{proof}
    Considering the natural surjection
    \begin{align*}
        \mathfrak{m}/\mathfrak{m}^2 \twoheadrightarrow \mathfrak{m}/(\mathfrak{m}^2 + (x_1,\ldots,x_n)),
    \end{align*}
    we have $\dim_k(\mathfrak{m}/\mathfrak{m}^2) = \dim_k(\mathfrak{m}/(\mathfrak{m}^2 + (x_1,\ldots,x_n))) + s$. Hence, we have
    \begin{align*}
        \embdim(A) = \embdim(A/(x_1,\ldots,x_n)) + s.
    \end{align*}
    On the other hand, since $x_1,\ldots,x_n$ form a regular sequence in $A$, we have
    \begin{align*}
        \dim(A) = \dim(A/(x_1,\ldots,x_n)) + n.
    \end{align*}
    Thus, we have 
    \begin{align*}
        \embdim(A) - \dim(A) = \embdim(A/(x_1,\ldots,x_n)) - \dim(A/(x_1,\ldots,x_n)) + s - n.
    \end{align*}
\end{proof}

In particular, we have the following corollary.

\begin{corollary}\label{cor:hs_regular_deform}
    Let $(A, \mathfrak{m}, k)$ be a Noetherian local ring. Then, the following statements hold:
    \begin{enumerate}
        \item Let $c \geq 0$, $n \geq 1$ and $x_1, \ldots, x_n \in \mathfrak{m}$ a regular sequence in $A$. If $A/(x_1, \ldots, x_n)$ is $\mathsf{CI}_{\leq c}$, then $A$ is $\mathsf{CI}_{\leq c}$ for any $c \geq 0$. 
        \item In (1), if $x_1, \ldots, x_n$ satisfy $x_i \notin \mathfrak{m}^2 + (x_1, \ldots, x_{i-1})$ for all $i$, then the converse holds.
    \end{enumerate}
\end{corollary}
\begin{proof}
    (1) is clear from Proposition \ref{prop:cibounded_regseq}. We show (2). It is enough to show that the following equivalence holds:
    \begin{itemize}
        \item[(i)] $x_i \notin \mathfrak{m}^2 + (x_1, \ldots, x_{i-1})$ for all $i \geq 1$.
        \item[(ii)] The images of $x_1, \ldots, x_n$ in $\mathfrak{m}/\mathfrak{m}^2$ are linearly independent over $k$.
    \end{itemize}
    (i) $\Longrightarrow$ (ii): Assume that the images of $x_1, \ldots, x_n$ in $\mathfrak{m}/\mathfrak{m}^2$ are not linearly independent over $k$. Then, there exist $a_1, \ldots, a_n \in A$ such that:
    \begin{align*}
        a_1 x_1 + \cdots + a_n x_n \in \mathfrak{m}^2,
    \end{align*}
    where $a_i \notin \mathfrak{m}$ for some $i$. Let $i_0$ be the largest integer such that $a_{i_0} \notin \mathfrak{m}$. Then, we have:
    \begin{align*}
        a_{i_0} x_{i_0} \in \mathfrak{m}^2 + (x_1, \ldots, x_{i_0-1}). 
    \end{align*}
    Since $a_{i_0}$ is a unit in $A$, we have:
    \begin{align*}
        x_{i_0} \in \mathfrak{m}^2 + (x_1, \ldots, x_{i_0-1}),
    \end{align*}
    which contradicts (i). Hence, the images of $x_1, \ldots, x_n$ in $\mathfrak{m}/\mathfrak{m}^2$ are linearly independent over $k$.

    (ii) $\Longrightarrow$ (i): Assume that $x_i \in \mathfrak{m}^2 + (x_1, \ldots, x_{i-1})$ for some $i$. Then, we have:
    \begin{align*}
        x_i = a_1 x_1 + \cdots + a_{i-1} x_{i-1} + b,
    \end{align*}
    where $a_1, \ldots, a_{i-1} \in A$ and $b \in \mathfrak{m}^2$. Hence, we have:
    \begin{align*}
        x_i - a_1 x_1 - \cdots - a_{i-1} x_{i-1} = b \in \mathfrak{m}^2.
    \end{align*}
    This contradicts (ii). Hence, $x_i \notin \mathfrak{m}^2 + (x_1, \ldots, x_{i-1})$ for all $i$.
\end{proof}

Hence, in general, we do not have the converse of the following implication:
\begin{align*}
    A/(x_1, \ldots, x_n) \text{ is } \mathsf{CI}_{\leq c} \Longrightarrow A \text{ is } \mathsf{CI}_{\leq c}.
\end{align*}
If we assume that the images of  $x_1, \ldots, x_n$ are linearly independent in $\mathfrak{m}/\mathfrak{m}^2$ over $k$, then we have the converse. Moreover, we want to take such a sequence as a system of parameters of $A$. Hence, we introduce the following definition.

\begin{definition}\label{def:linearly_indep}
    Let $(A, \mathfrak{m}, k)$ be a Noetherian local ring with $d = \dim(A)$. A system of parameters $x_1, \ldots, x_d \in \mathfrak{m}$ of $A$ is called a \emph{linearly independent system of parameters} if $x_i \notin \mathfrak{m}^2 + (x_1, \ldots, x_{i-1})$ for all $i \geq 1$.
\end{definition}

Let $(A, \mathfrak{m}, k)$ be a Noetherian local ring. In general, if $x_1, \ldots, x_d \in \mathfrak{m}$ form a regular sequence, then $x_1^{\nu_1}, \ldots, x_d^{\nu_d}$ also form a regular sequence for any positive integers $\nu_1, \ldots, \nu_d$. Under these circumstances, a linearly independent system of parameters is a ``good'' system of parameters from the viewpoint of embedding dimension. 

\begin{lemma}\label{lem:linearly_indep}
    Let $(A, \mathfrak{m}, k)$ be a Cohen--Macaulay local ring with $d = \dim(A)$. Then, there exists a linearly independent system of parameters (Definition \ref{def:linearly_indep}) $x_1, \ldots, x_d \in \mathfrak{m}$ of $A$.
\end{lemma}
\begin{proof}
    If $d=0$, then the empty sequence is a linearly independent system of parameters of $A$. Hence, we may assume that $d > 0$. We show this lemma by induction on $d = \dim(A)$. 

    ($d=1$) In general, for any Noetherian ring $R$, the set $D$ of all zero divisors of $R$ is the union of all associated prime ideals of $R$, i.e.,
    \begin{align*}
        D = \bigcup_{\mathfrak{p} \in \Ass(R)} \mathfrak{p}.
    \end{align*}
    If $\mathfrak{m} \in \Ass(A)$, $D = \mathfrak{m}$. Hence, we cannot take a regular element in $\mathfrak{m}$. Therefore, $\depth(A) = 0$. This contradicts the Cohen--Macaulayness of $A$. Hence, $\mathfrak{m} \notin \Ass(A)$. In particular, $\mathfrak{m} \nsubseteq \mathfrak{p}$ for any $\mathfrak{p} \in \Ass(A)$. Moreover, $\mathfrak{m} \nsubseteq \mathfrak{m}^2$ by Nakayama's lemma. Thus, by the prime avoidance lemma, we can take an element $x_1 \in A$ such that:
    \begin{align*}
        x_1 \in \mathfrak{m} \setminus \left( \mathfrak{m}^2 \cup \bigcup_{\mathfrak{p} \in \Ass(A)} \mathfrak{p}\right).
    \end{align*}
    Then, $x_1$ is a regular element in $A$ and $x_1 \notin \mathfrak{m}^2$. Moreover, since $\dim(A) = 1$, $x_1$ is a system of parameters of $A$ by Krull's principal ideal theorem. Hence, $x_1$ is a linearly independent system of parameters of $A$.

    ($d>1$) By the same argument as the case $d=1$, we can take a regular element $x_1 \in \mathfrak{m} \setminus \mathfrak{m}^2$. Then, $\dim(A/(x_1)) = d-1$ and $A/(x_1)$ is Cohen--Macaulay. By the induction hypothesis, there exists a linearly independent system of parameters $\overline{x_2}, \ldots, \overline{x_d} \in \mathfrak{m}/(x_1)$ of $A/(x_1)$, where each $\overline{x_i}$ is the image of $x_i$ in $A/(x_1)$. Then, $x_1, x_2, \ldots, x_d \in \mathfrak{m}$ form a linearly independent system of parameters of $A$. It is clear that these elements form a system of parameters. We have to show that $x_i \notin \mathfrak{m}^2 + (x_1, \ldots, x_{i-1})$ for all $i$. If $i=1$, it is clear by the choice of $x_1$. If $i>1$, we assume that $x_i \in \mathfrak{m}^2 + (x_1, \ldots, x_{i-1})$. Then, taking the image of $x_i$ in $A/(x_1)$, we have:
    \begin{align*}
        \overline{x_i} \in ((\mathfrak{m}^2+(x_1))/(x_1)) + (\overline{x_2}, \ldots, \overline{x_{i-1}}). 
    \end{align*}
    This contradicts the linearly independence of $\overline{x_2}, \ldots, \overline{x_d}$ in $(\mathfrak{m}/(x_1))/((\mathfrak{m}^2+(x_1))/(x_1))$ by the proof of Corollary \ref{cor:hs_regular_deform}. Hence, $x_i \notin \mathfrak{m}^2 + (x_1, \ldots, x_{i-1})$ for all $i$. Therefore, we have shown that there exists a linearly independent system of parameters $x_1, \ldots, x_d \in \mathfrak{m}$ of $A$.
\end{proof}

The following proposition guarantees that $\mathsf{CI}_{\leq c}$ is also a good property in the sense of (A-1) and (A-2) in Definition \ref{def:conditions-for-NC}.

\begin{proposition}\label{prop:nagata_hs}
    Let $c$ be a non-negative integer and $\PP = \mathsf{CI}_{\leq c}$ the property of local rings defined in Definition \ref{def:cicodim}. Then, $\PP$ satisfies \textup{(A-1)} and \textup{(A-2)} in Definition \ref{def:conditions-for-NC}. 
\end{proposition}
\begin{proof}
    (A-1) Let $(A, \mathfrak{m})$ be a local complete intersection ring of codimension at most $c$ and $\mathfrak{p}$ be a prime ideal of $A$. Then, $A_\mathfrak{p}$ is also a local complete intersection ring of codimension at most $c$; see \cite{Put21}. 

    (A-2) Let $A$ be a Noetherian ring and $P \in \mathsf{CI}_{\leq c}(A)$. Then, $A_P$ is a complete intersection ring of codimension at most $c$. In particular, $A_P$ is Cohen--Macaulay. We can take $x_1, \ldots, x_n \in P$ such that $x_1, \ldots, x_n$ form a linearly independent system of parameters of $A_P$ by Lemma \ref{lem:linearly_indep}. Hence, (A-2-i) holds.  Moreover, (A-2-ii) is also satisfied because $A_P/(x_1, \ldots, x_n)A_P$ is a complete intersection of codimension at most $c$ by Corollary \ref{cor:hs_regular_deform} (2). We check (A-2-iii). We can take an open neighbourhood $U$ of $P$ in $\Spec(A)$ such that, for every $Q \in U$, $x_1, \ldots, x_n$ form a regular sequence in $A_Q$. Then, by Corollary \ref{cor:hs_regular_deform} (1), we have that $A_Q/(x_1, \ldots, x_n)A_Q$ is $\mathsf{CI}_{\leq c} \Longrightarrow$ $A_Q$ is $\mathsf{CI}_{\leq c}$ for every $Q \in U$. Hence, (A-2-iii) holds. 
\end{proof}

Moreover, $\mathsf{CI}_{\leq c}$ is stable under polynomial extensions. 

\begin{proposition}\label{prop:polyext}
    Let $A$ be a Noetherian ring. Then, the following statements hold for any integer $c \geq 0$:
    \begin{enumerate}
        \item $A$ is $\mathsf{CI}_{\leq c}$ $\Longleftrightarrow$ $A[x]$ is $\mathsf{CI}_{\leq c}$.
        \item $A$ is $\mathsf{CI}_{\leq c}$ $\Longleftrightarrow$ $A[[x]]$ is $\mathsf{CI}_{\leq c}$.
    \end{enumerate}
\end{proposition}
\begin{proof}
    (1) First, we see $\Longleftarrow$. Let $\mathfrak{p}$ be any prime ideal of $A$. We have to show that $A_\mathfrak{p}$ is $\mathsf{CI}_{\leq c}$. We put $\mathfrak{P} = \mathfrak{p}A[x] + (x)$. Indeed, $\mathfrak{P}$ is a prime ideal of $A[x]$ because $A[x]/\mathfrak{P} \cong A/\mathfrak{p}$. We have
    \begin{align*}
        A[x]_{\mathfrak{P}}/xA[x]_{\mathfrak{P}} \cong (A[x]/(x))_{\mathfrak{P}/(x)} \cong A_\mathfrak{p}.
    \end{align*}
    Hence, it is enough to show that $x$ is a regular element with $x \notin \mathfrak{P}^2 A[x]_{\mathfrak{P}}$ by Corollary \ref{cor:hs_regular_deform} (2). It is a basic fact that $x$ is a regular element in $A[x]_{\mathfrak{P}}$. Moreover, we have
    \begin{align*}
        A[x]_{\mathfrak{P}}/\mathfrak{p}A[x]_{\mathfrak{P}} \cong \kappa(\mathfrak{p})[x]_{(x)}. 
    \end{align*}
    Since $x \notin (x)^2 \kappa(\mathfrak{p})[x]_{(x)}$, we have $x \notin \mathfrak{P}^2 A[x]_{\mathfrak{P}}$. Second, we see $\Longrightarrow$. Let $\mathfrak{P}$ be any prime ideal of $A[x]$. It is well known that $A[x]$ is $\mathsf{CI}$. Hence, it is enough to show that $\embcodim(A[x]_{\mathfrak{P}}) \leq c$. We put $\mathfrak{p} = \mathfrak{P} \cap A$. Then, 
    \begin{align*}
        A_\mathfrak{p} \rightarrow A[x]_{\mathfrak{P}}
    \end{align*}
    is a flat local homomorphism. The closed fiber is given by
    \begin{align*}
        A[x]_{\mathfrak{P}}/\mathfrak{p}A[x]_{\mathfrak{P}} &\cong A[x]/\mathfrak{p}A[x] \otimes_{A[x]} A[x]_{\mathfrak{P}} \\ 
        &\cong A/\mathfrak{p} \otimes_{A} A[x] \otimes_{A[x]} A[x]_{\mathfrak{P}} \\
        &\cong A/\mathfrak{p} \otimes_{A} A_{\mathfrak{p}}[x]_{\mathfrak{P}A_{\mathfrak{p}}[x]} \\
        &\cong \kappa(\mathfrak{p}) \otimes_{A_{\mathfrak{p}}} A_{\mathfrak{p}}[x]_{\mathfrak{P}A_{\mathfrak{p}}[x]} \\ 
        &\cong \kappa(\mathfrak{p})[x]_{\mathfrak{P}\kappa(\mathfrak{p})[x]}.
    \end{align*}
    Hence, by \cite[Lemma 3.1]{NSW15}, we have 
    \begin{align*}
        \embdim(A_\mathfrak{p}) + \embdim(\kappa(\mathfrak{p})[x]_{\mathfrak{P}\kappa(\mathfrak{p})[x]}) = \embdim(A[x]_{\mathfrak{P}})
    \end{align*}
    Therefore, we have
    \begin{align*}
        \embcodim(A[x]_{\mathfrak{P}}) &= \embdim(A[x]_{\mathfrak{P}}) - \dim(A[x]_{\mathfrak{P}}) \\
        &= \embdim(A_\mathfrak{p}) + \embdim(\kappa(\mathfrak{p})[x]_{\mathfrak{P}\kappa(\mathfrak{p})[x]}) - (\dim(A_\mathfrak{p}) + \dim(\kappa(\mathfrak{p})[x]_{\mathfrak{P}\kappa(\mathfrak{p})[x]})) \\
        &= \embcodim(A_\mathfrak{p}) \\ 
        &\leq c.
    \end{align*}
    Here, we used the fact that $\embdim(\kappa(\mathfrak{p})[x]_{\mathfrak{P}\kappa(\mathfrak{p})[x]}) = \dim(\kappa(\mathfrak{p})[x]_{\mathfrak{P}\kappa(\mathfrak{p})[x]})$ and the dimension equality for flatness; see \cite[Theorem 15.1]{Mat86}. 

    (2) $\Longleftarrow$ is similar to $\Longleftarrow$ in (1). We show $\Longrightarrow$. Let $\mathfrak{M}$ be any maximal ideal of $A[[x]]$. We put $\mathfrak{m} = \mathfrak{M} \cap A$. It is enough to show that $A[[x]]_{\mathfrak{M}}$ is $\mathsf{CI}_{\leq c}$ because $\mathsf{CI}_{\leq c}$ is stable under generalization. We have $\mathfrak{M} = \mathfrak{m}A[[x]] + (x)$; see \cite[Exercise 1.5 iv)]{AM69}. We have
    \begin{align*}
        \hat{A[[x]]_{\mathfrak{M}}} \cong \hat{A_\mathfrak{m}}[[x]];
    \end{align*}
    see the proof of \cite[Theorem 19.5]{Mat86}. Hence, we have
    \begin{align*}
        \embcodim(A[[x]]_{\mathfrak{M}}) &= \embcodim(\hat{A[[x]]_{\mathfrak{M}}}) \\
        &= \embcodim(\hat{A_\mathfrak{m}}[[x]]) \\
        &= \embdim(\hat{A_\mathfrak{m}}[[x]]) - \dim(\hat{A_\mathfrak{m}}[[x]]) \\
        &= \embdim(\hat{A_\mathfrak{m}}) + 1 - (\dim(\hat{A_\mathfrak{m}}) + 1) \\
        &= \embcodim(A_\mathfrak{m}) \\ 
        &\leq c.
    \end{align*}
    This is what we want to show.
\end{proof}

In particular, hypersurfaces can be characterized by the following numerical characterization. A similar statement can be found in \cite[5.1]{Avr98}. This seems to be a well-known fact, but the author could not find an explicit reference. Hence, we give a proof for the sake of completeness.

\begin{proposition}\label{prop:indep_hs}
    Let $(A, \mathfrak{m}, k)$ be a Noetherian local ring. Then, the following conditions are equivalent:
    \begin{enumerate}
        \item $\hat{A} = R/(f)$ for some regular local ring $(R, \mathfrak{n})$ and a regular element $f$ in $R$.
        \item $A$ is a hypersurface ring. 
        \item $A$ is a Cohen--Macaulay ring and $\embcodim(A) \leq 1$. 
        \item $\embdim(A) - \depth(A) \leq 1$. 
    \end{enumerate}
\end{proposition}
\begin{proof}
    (1) $\Longrightarrow$ (2): Since $A$ is $\mathsf{CI}$, we have to show that $\embcodim(A) \leq 1$. We have
    \begin{align*}
        \embdim(A) - \dim(A) &= \embdim(\hat{A}) - \dim(\hat{A}) \\ 
        &= \embdim(R/(f)) - \dim(R/(f)) \\ 
        &\leq \embdim(R) - (\dim(R) - 1) \\ 
        &= 1.
    \end{align*}

    (2) $\Longrightarrow$ (3): Clear. 

    (3) $\Longrightarrow$ (4): Since $A$ is Cohen--Macaulay, we have $\depth(A) = \dim(A)$. Hence, we have $\embdim(A) - \depth(A) = \embdim(A) - \dim(A) \leq 1$.

    (4) $\Longrightarrow$ (1): If $\embdim(A) - \depth(A) = 0$, $A$ is a regular local ring. Hence, we can check (1). We assume $\embdim(A) - \depth(A) = 1$. We take a minimal Cohen presentation $\hat{A} = R/I$. Then, we have:
    \begin{align*}
        \embdim(A) &= \dim_k(\mathfrak{m}/\mathfrak{m}^2) \\ 
        &= \dim_k((\mathfrak{n}/I)/(\mathfrak{n}^2/I)) \\ 
        &= \dim_k(\mathfrak{n}/(\mathfrak{n}^2 + I)) \\ 
        &= \dim_k(\mathfrak{n}/\mathfrak{n}^2) \\ 
        &= \embdim(R) \\ 
        &= \dim(R).
    \end{align*}
    On the other hand, $R/I$ has a finite projective dimension as an $R$-module because $R$ is a regular local ring. Hence, by the Auslander--Buchsbaum formula, we have:
    \begin{align*}
        \projdim_R(R/I) &= \depth_R(R) - \depth_R(R/I) \\ 
        &= \dim(R) - \depth_{R/I}(R/I) \\ 
        &= \embdim(A) - \depth(\hat{A}) \\
        &= \embdim(A) - \depth(A) \\
        &= 1.
    \end{align*}
    Consider the natural exact sequence
    \begin{align*}
        0 \longrightarrow I \longrightarrow R \longrightarrow R/I \longrightarrow 0. 
    \end{align*}
    For any $R$-module $M$, applying the functor $\Hom_R(-, M)$ to the above exact sequence yields the following long exact sequence:
    \[\begin{tikzcd}[sep=tiny]
        0 & {\Hom_R(R/I, M)} & {\Hom_R(R, M)} & {\Hom_R(I, M)} & \hspace{0pt} \\
        \hspace{0pt} & {\Ext_R^1(R/I, M)} & {\Ext_R^1(R, M)} & {\Ext_R^1(I, M)} & \hspace{0pt}  \\
        \hspace{0pt} & {\Ext_R^2(R/I, M)} & {\Ext_R^2(R, M)} & {\Ext_R^2(I, M)} & \cdots. 
        \arrow[from=1-1, to=1-2]
        \arrow[from=1-2, to=1-3]
        \arrow[from=1-3, to=1-4]
        \arrow[from=1-4, to=1-5]
        \arrow[from=2-1, to=2-2]
        \arrow[from=2-2, to=2-3]
        \arrow[from=2-3, to=2-4]
        \arrow[from=2-4, to=2-5]
        \arrow[from=3-1, to=3-2]
        \arrow[from=3-2, to=3-3]
        \arrow[from=3-3, to=3-4]
        \arrow[from=3-4, to=3-5]
    \end{tikzcd}\]
    Then, we have $\Ext_R^1(I, M) = \Ext_R^2(R/I, M) = 0$ because $R$ is projective as an $R$-module and $\projdim_R(R/I) = 1$. Hence, $I$ is projective as an $R$-module. Since $R$ is a local ring, $I$ is free as an $R$-module. Because of the natural injection $I \hookrightarrow R$, $\rank(I) = 1$ as an $R$-module. Hence, $I = (f)$ for some regular element $f \in R$. Therefore, we have $\hat{A} = R/(f)$. So, we have shown that (4) $\Longrightarrow$ (1).
\end{proof}

(3) $\Longrightarrow$ (1) cannot be generalized to complete intersection of codimension at most $c$ for $c \geq 2$. This is because there exists a Noetherian local ring $(A, \mathfrak{m}, k)$ such that $A$ is Cohen--Macaulay but is not a complete intersection. 

At the end of this subsection, we give the following lemma which plays an important role when we give a counterexample to (A-3) in Definition \ref{def:conditions-for-NC} for $\mathsf{CI}_{\leq c}$ in Theorem \ref{th:counterexampleNC}.

\begin{lemma}\label{lem:squarezero}
    Let $(A, \mathfrak{m}, k)$ be a Noetherian local ring. We put $B = A[x]/(x^2)$. Then, the following statements hold:
    \begin{enumerate}
        \item $(A, \mathfrak{m}, k)$ is a complete intersection local ring of codimension $c$ if and only if $A[x]/(x^2)$ is a complete intersection local ring of codimension $c+1$.
        \item $\mathsf{CI}_{\leq c-1}(A) = \mathsf{CI}_{\leq c}(B)$ for any integer $c \geq 1$.
    \end{enumerate}
\end{lemma}
\begin{proof}
    (1) Since $\hat{B} = \hat{A}[[x]]/(x^2)$, $A$ is a complete intersection if and only if $B$ is a complete intersection. Moreover, we have
    \begin{align*}
        \embdim(B) = \dim_k((\mathfrak{m}/\mathfrak{m}^2) \oplus kx) = \embdim(A) + 1
    \end{align*}
    and $\dim(B) = \dim(A)$. Hence, we have $\embcodim(B) = \embcodim(A) + 1$. Therefore, $A$ is a complete intersection of codimension $c$ if and only if $B$ is a complete intersection of codimension $c+1$.

    (2) The following natural map is a homeomorphism:
    \begin{align*}
        \Spec(B) \longrightarrow \Spec(A), \quad \mathfrak{P} \longmapsto \mathfrak{P} \cap A.
    \end{align*}
    The inverse map is given by $\mathfrak{p} \longmapsto \mathfrak{p}B + (x)$. Moreover, we have $B_{\mathfrak{P}} = A_{\mathfrak{p}}[t]/(t^2)$. Hence, by (1), we have
    \begin{align*}
        \mathfrak{p} \in \mathsf{CI}_{\leq c-1}(A) &\Longleftrightarrow A_\mathfrak{p} \text{ is } \mathsf{CI}_{\leq c-1} \\
        &\Longleftrightarrow A_\mathfrak{p}[t]/(t^2) \text{ is } \mathsf{CI}_{\leq c} \\
        &\Longleftrightarrow B_{\mathfrak{P}} \text{ is } \mathsf{CI}_{\leq c} \\
        &\Longleftrightarrow \mathfrak{P} \in \mathsf{CI}_{\leq c}(B).
    \end{align*}
\end{proof}

\subsection{Failure of the Nagata criterion}\label{subsec:failure-nagata}

We have seen that $\mathsf{CI}_{\leq c}$ shares many properties with the other classes of singularities. In particular, (A-1), (A-2) and (A-3) imply \NC{} by Proposition \ref{prop:conditions-imply-NC}, and $\mathsf{CI}_{\leq c}$ satisfies (A-1) and (A-2) by Proposition \ref{prop:nagata_hs}. However, in this subsection, we see that $\mathsf{CI}_{\leq c}$ does not satisfy (A-3) for $c>0$ in general. Therefore, $\mathsf{CI}_{\leq c}$ does not satisfy \NC{} for $c>0$ in general. First, we show that $\mathsf{HS} = \mathsf{CI}_{\leq 1}$ does not satisfy (A-3) in Definition \ref{def:conditions-for-NC} in general.

\begin{theorem}\label{th:counterexampleNC}
    For every $c \geq 1$, the property $\mathsf{CI}_{\leq c}$ does not satisfy \NC{}.   
\end{theorem}
\begin{proof}
    We first treat the case $c=1$. By Proposition \ref{prop:nagata_hs}, $\mathsf{HS}=\mathsf{CI}_{\leq 1}$ satisfies (A-1) and (A-2). Hence, by Proposition \ref{prop:conditions-imply-NC}, it is enough to show that $\mathsf{HS}$ does not satisfy (A-3).

    We have to construct a Noetherian ring $A$ which is $\mathsf{HS}$-Q0 and a prime ideal $P$ of $A$ such that $A_P$ is a hypersurface, $P^r = 0$ for some $r \geq 1$, $P^i/P^{i+1}$ is a free $A/P$-module for all $i \geq 0$ and we cannot find an open neighbourhood $U$ of $P$ in $\Spec(A)$ such that, for every $Q \in U$, if $A_Q/PA_Q$ is a hypersurface, then $A_Q$ is a hypersurface. 

    Fix a purely transcendental extension $K/K_0$ of countably infinite transcendence degree, where $K_0$ is a countable field. For instance, take $K = \QQ(x_1, x_2, \ldots)$. By \cite[Example 2.7, Example 2.11]{Nis12}, there exists a hypersurface local domain $B$ containing $K$ whose regular locus $\mathsf{Reg}(B)$ does not contain any nonempty open subset of $\Spec(B)$. Moreover, $B/\mathfrak{p}$ is essentially of finite type over $K$ for any $\mathfrak{p} \in \Spec(B) \setminus \{(0)\}$. We put $A = B[t]/(t^2)$ and $P = (t)$. For any $\mathfrak{p} \in \Spec(A)$, we can take a prime ideal $\mathfrak{q} \in B$ such that $\mathfrak{p} = \mathfrak{q}A + (t)$. Hence, we have $A/\mathfrak{p} = B/\mathfrak{q}$. If $\mathfrak{q} = 0$, then $A/\mathfrak{p} = B$ is a hypersurface. If $\mathfrak{q} \neq 0$, then $B/\mathfrak{q}$ is essentially of finite type over $K$. Hence, $A/\mathfrak{p}$ is an excellent domain. In particular, $\mathsf{Reg}(A/\mathfrak{p})$ is a nonempty open subset of $\Spec(A/\mathfrak{p})$. Therefore, $A$ is $\mathsf{HS}$-Q0. Moreover, we have $P^2 = 0$ and $P/P^2 \cong A/P$ as an $A/P$-module. Hence, $P^i/P^{i+1}$ is a free $A/P$-module for all $i \geq 0$. For any $Q \in \Spec(A)$ with $Q = \mathfrak{q}A + (t)$, we have $A_Q/PA_Q \cong (A/P)_{(Q/P)} \cong B_\mathfrak{q}$, which is a hypersurface. Assume that there exists an open neighbourhood $U$ of $P$ in $\Spec(A)$ such that, for every $Q \in U$, if $A_Q/PA_Q$ is a hypersurface, then $A_Q$ is a hypersurface. Since $\mathsf{Reg}(B)$ does not contain any nonempty open subset of $\Spec(B)$, we can take a prime ideal $\mathfrak{q} \in \Spec(B) \setminus \mathsf{Reg}(B)$ such that $Q = \mathfrak{q}A + (t) \in U$. Then, we have $A_Q = (B[t]/(t^2))_{\mathfrak{q}A + (t)} = B_{\mathfrak{q}}[t]/(t^2)$. Since $B_{\mathfrak{q}}[t]/(t^2)$ has codimension $2$ by Lemma \ref{lem:squarezero}, $A_Q$ is not a hypersurface. This contradicts the assumption. Hence, $\mathsf{HS}$ does not satisfy (A-3) in Definition \ref{def:conditions-for-NC} in general. In fact, the argument shows that $A$ satisfies $\mathsf{HS}$-Q0 and that $\mathsf{HS}(A)$ is not open in $\Spec(A)$. Thus, $A$ is a counterexample to \NC{} for $\mathsf{HS}$.
    
    Next, we show that $\mathsf{CI}_{\leq c}$ does not satisfy \NC{} in general for any $c \geq 1$. We now proceed by induction on $c$. The case $c = 1$ is already shown. Suppose that $B$ is a counterexample to \NC{} for $\mathsf{CI}_{\leq c-1}$ and put $A = B[t]/(t^2)$. Since $B$ satisfies $\mathsf{CI}_{\leq c-1}$-Q0, $A$ also satisfies $\mathsf{CI}_{\leq c-1}$-Q0 by Proposition \ref{prop:Q0preserved}. Moreover, since $\mathsf{CI}_{\leq c-1} \implies \mathsf{CI}_{\leq c}$, $A$ satisfies $\mathsf{CI}_{\leq c}$-Q0. However, $\mathsf{CI}_{\leq c}(A) = \mathsf{CI}_{\leq c-1}(B)$ (Lemma \ref{lem:squarezero} (2)) is not open in $\Spec(A)$ because $\mathsf{CI}_{\leq c-1}(B)$ is not open in $\Spec(B)$. Hence, $A$ is a counterexample to \NC{} for $\mathsf{CI}_{\leq c}$.
\end{proof}

\begin{remark}\label{rem:counterexring}
    According to \cite{Nis12}, if we use Example 2.7 in the proof of Theorem \ref{th:counterexampleNC}, then $B$ is a two-dimensional Noetherian local domain. We can prove that this ring is not a Nagata ring. On the other hand, if we take $B$ from Example 2.11, then $B$ is a three-dimensional Noetherian local Nagata domain. Hence, the counterexample $A$ in the proof of Theorem \ref{th:counterexampleNC} can be taken as a three-dimensional Noetherian local Nagata ring.
\end{remark}

\subsection{Recovery of openness}\label{subsec:recovery}

We have seen that $\mathsf{CI}_{\leq c}$ does not satisfy \NC{} in general for any $c \geq 1$. For singularities which satisfy \NC{}, such as $\mathsf{Reg}$, $\mathsf{CI}$, $\mathsf{Gor}$, and $\mathsf{CM}$, their loci are open in $\Spec(A)$ for any $\mathsf{Reg}$-Q0 ring $A$. In other words, let $A$ be a Noetherian ring which satisfies the following condition:
\begin{align*}
    \forall \mathfrak{p} \in \Spec(A), \mathsf{Reg}(A/\mathfrak{p}) \text{ contains a nonempty open subset of } \Spec(A/\mathfrak{p}).
\end{align*}
Then, $\mathsf{Reg}(A),\ \mathsf{CI}(A),\ \mathsf{Gor}(A)$, and $\mathsf{CM}(A)$ are open subsets of $\Spec(A)$. This is easily shown by \NC{}. Since $\mathsf{CI}_{\leq c}$ does not satisfy \NC{}, we cannot deduce the openness of $\mathsf{CI}_{\leq c}(A)$ for every $\mathsf{Reg}$-Q0 ring $A$ by the same argument. However, we can show that $\mathsf{CI}_{\leq c}(A)$ is an open subset of $\Spec(A)$ for any $\mathsf{Reg}$-Q0 ring $A$. We first recall a lifting result for regular sequences across a normally flat nilpotent thickening. 

\begin{lemma}[{\cite[Exercise 24.1]{Mat86}}]\label{lem:normallyflat}
    Let $A$ be a Noetherian ring, $x_1, \ldots, x_n$ a sequence of elements of $A$ and $I$ an ideal of $A$ such that $I^r = 0$ for some $r \geq 1$. Assume that $I^i/I^{i+1}$ is a free $A/I$-module for all $i \geq 0$. Then, $x_1, \ldots, x_n$ is an $A$-regular sequence if and only if $x_1, \ldots, x_n$ is an $A/I$-regular sequence.
\end{lemma}

\begin{theorem}\label{th:openness_hs}
    Let $A$ be a Noetherian ring which satisfies $\mathsf{Reg}$\textup{-Q0}. Then, $\mathsf{CI}_{\leq c}(A)$ is an open subset of $\Spec(A)$ for every $c \in \NN$.
\end{theorem}
\begin{proof}
    By repeating the proof of Proposition \ref{prop:conditions-imply-NC} for the fixed ring $A$, it is enough to verify the following analogue of (A-3). More precisely, it is enough to show that, for any Noetherian ring $A$ which satisfies $\mathsf{Reg}$-Q0, any prime ideal $P$ of $A$ such that $A_P$ is $\mathsf{CI}_{\leq c}$, and any integer $r \geq 1$ such that $P^r = 0$ and $P^i/P^{i+1}$ is a free $A/P$-module for all $i \geq 0$, there exists an open neighbourhood $U$ of $P$ in $\Spec(A)$ such that, for every $Q \in U$, if $A_Q/PA_Q$ is $\mathsf{CI}_{\leq c}$, then $A_Q$ is $\mathsf{CI}_{\leq c}$. Since $A$ is $\mathsf{Reg}$-Q0, we can take an open neighbourhood $U_1$ of $P$ in $\Spec(A)$ such that $A_Q/PA_Q$ is regular for every $Q \in U_1$. Moreover, since $\mathsf{CI}(A)$ is open in $\Spec(A)$, we can take an open neighbourhood $U_2$ of $P$ in $\Spec(A)$ such that, for every $Q \in U_2$, $A_Q$ is $\mathsf{CI}$. We put $U = U_1 \cap U_2$. We have to show that $A_Q$ is $\mathsf{CI}_{\leq c}$ for any $Q \in U$. Fix $Q \in U$. Since $P \in \mathsf{CI}_{\leq c}(A)$, we have
    \begin{align*}
        \embdim(A_P) - \dim(A_P) \leq c.
    \end{align*}
    Since $A_P$ is zero-dimensional, we have $\dim_{\kappa(P)}PA_P/P^2A_P \leq c$. On the other hand, $P/P^2 = (A/P)^{\oplus e}$ for some $e \geq 0$. Tensoring with $A_Q$ over $A$, we obtain
    \begin{align*}
        PA_Q/P^2A_Q = (A_Q/PA_Q)^{\oplus e}.
    \end{align*}
    Taking $Q = P$, we have $\dim_{\kappa(P)}PA_P/P^2A_P = e$. Hence, we have $e \leq c$. Moreover, we have
    \begin{align*}
        PA_Q &= z_1 A_Q + \cdots + z_e A_Q + P^2A_Q \\ 
        &= z_1 A_Q + \cdots + z_e A_Q + P^3A_Q \\
        &\vdots \\
        &= z_1 A_Q + \cdots + z_e A_Q. \\
    \end{align*}
    Hence, we have $\mu_{A_Q}(PA_Q) \leq e \leq c$. Choose $x_1, \ldots, x_d$ $(d = \height(Q))$ in $A_Q$ such that their images in $A_Q/PA_Q$ form a regular system of parameters. By Lemma \ref{lem:normallyflat}, we find that $x_1, \ldots, x_d$ form a regular sequence in $A_Q$. Let $B = A_Q/(x_1, \ldots, x_d)A_Q$. We have $\dim(B) = 0$ and $(x_1, \ldots, x_d) + PA_Q = QA_Q$. Then, 
    \begin{align*}
        \mathfrak{m}_B = \frac{PA_Q + (x_1, \ldots, x_d)}{(x_1, \ldots, x_d)}
    \end{align*}
    is a maximal ideal of $B$. We also have $\mu_B(\mathfrak{m}_B) \leq c$. Thus, $B$ is a complete intersection local ring with
    \begin{align*}
        \embdim(B) - \dim(B) = \embdim(B) \leq c.
    \end{align*}
    Therefore, Corollary \ref{cor:hs_regular_deform} (1) shows that $A_Q$ is $\mathsf{CI}_{\leq c}$. 
\end{proof}

Clearly, an excellent ring satisfies $\mathsf{Reg}$-Q0. Hence, we have the following corollary. This is analogous to \cite[Corollary 3.3]{GM78}.

\begin{corollary}\label{cor:exopen1}
    Let $A$ be an excellent ring. Then, $\mathsf{CI}_{\leq c}(A)$ is an open subset of $\Spec(A)$ for any $c \geq 0$.
\end{corollary}

In particular, since $\mathsf{CI}(A) = \bigcup_{c \geq 0} \mathsf{CI}_{\leq c}(A)$, we can recover Corollary 3.3 of \cite{GM78} from Theorem \ref{th:openness_hs}.

\section{The codimension function on excellent schemes}\label{sec:constructible}

We study the codimension function on excellent schemes and apply its constructibility to complete intersections of bounded codimension. Briefly, we discuss the following map for a scheme $X$:
\begin{align*}
    C \colon X \to \NN \colon x \mapsto \embdim(\mathcal{O}_{X, x}) - \dim(\mathcal{O}_{X, x}).
\end{align*}
As a result, we give another geometric proof of Corollary \ref{cor:exopen1}. 

\begin{definition}\label{def:uppersemiconti}
    Let $X$ be a topological space and $F \colon X \rightarrow \NN$ be a function. We say that $F$ is \emph{upper semicontinuous} on $X$ if, for any $n \in \NN$, the set $F^{-1}(\NN_{\leq n}) = \{x \in X \mid F(x) \leq n\}$ is open in $X$.
\end{definition}

The upper semicontinuity is checked by an open covering of $X$.

\begin{lemma}\label{lem:uppersemiconti_opencover}
    Let $X$ be a topological space and $F \colon X \rightarrow \NN$ be a function. Fix an open covering $\{U_i\}_{i \in I}$ of $X$. $F$ is upper semicontinuous on $X$ if and only if $F|_{U_i}$ is upper semicontinuous on $U_i$ for all $i \in I$. 
\end{lemma}
\begin{proof}
    Assume that $F$ is upper semicontinuous on $X$. Then, for any $n \in \NN$, we have $F^{-1}(\NN_{\leq n})$ is open in $X$. Hence, $F|_{U_i}^{-1}(\NN_{\leq n}) = F^{-1}(\NN_{\leq n}) \cap U_i$ is open in $U_i$ for all $i \in I$. Therefore, $F|_{U_i}$ is upper semicontinuous on $U_i$ for all $i \in I$.

    Assume that $F|_{U_i}$ is upper semicontinuous on $U_i$ for all $i \in I$. Then, for any $n \in \NN$, we have $F|_{U_i}^{-1}(\NN_{\leq n})$ is open in $U_i$ for all $i \in I$. Hence, we have
    \begin{align*}
        F^{-1}(\NN_{\leq n}) = \bigcup_{i \in I} F|_{U_i}^{-1}(\NN_{\leq n}),
    \end{align*}
    which is open in $X$. Therefore, $F$ is upper semicontinuous on $X$.
\end{proof}

Ragusa, Tarrio, Rodicio, and Avramov studied complete intersections using Andr\'{e}--Quillen homology. They defined various invariants of Noetherian local rings using Andr\'{e}--Quillen homology and proved that these invariants detect complete intersections. In particular, they studied the upper semicontinuity of mappings derived from these invariants. 

\begin{definition}\label{def:delta}
    Let $(A, \mathfrak{m}, k)$ be a Noetherian local ring. 
    \begin{enumerate}
        \item The \emph{deviations} $\varepsilon_n(A)$ are defined by the formal power series identity:
        \begin{align*}
            \sum_{i \in \NN} (\dim_k \Tor_i^A(k,k)) t^i = \prod_{i \in \NN} (1 - (-t)^{i+1})^{(-1)^i \varepsilon_i(A)};
        \end{align*}
        see \cite{Avr77} or \cite[CHAPTER THREE \S 1]{GL69}. 
        \item We define the \emph{complete intersection defect} $d(A)$ of $A$ by
        \begin{align*}
            d(A) = \dim(A) - \varepsilon_0(A) + \varepsilon_1(A);
        \end{align*}
        see \cite{ATR89}. 
        \item For $n \in \NN$, we define the \emph{deviation} $\delta_n(A)$ by
        \begin{align*}
            \delta_n(A) = \dim_k H_n(A, k, k),
        \end{align*}
        where $H_n(A, k, k)$ is Andr\'{e}--Quillen homology; see \cite[DEFINITION 1.5]{Rag80}.
    \end{enumerate}
\end{definition}

\begin{remark}\label{rem:invariants}
    The definitions of the invariant $\varepsilon_n(A)$ in \cite{Avr77}, \cite{GL69}, \cite{ATR89}, and \cite{Rag80} are equivalent. Moreover, according to \cite[REMARK 1.8]{Rag80}, if $n \leq \pi(A)$, then we have $\delta_n(A) = \varepsilon_{n-1}(A)$, where
    \begin{align*}
        \pi(A) = 
        \begin{dcases}
            \infty & (\operatorname{Char} k = 0) \\
            2p & (\operatorname{Char} k = p > 0)
        \end{dcases}
    \end{align*}
    for a Noetherian local ring $(A, \mathfrak{m}, k)$. In particular, we always have $\delta_2(A) = \varepsilon_1(A)$.

    By \cite[LEMMA 3.1.2 (ii)]{GL69}, we have $\varepsilon_0(A) = \embdim(A)$. As we can see from this, the definition of $\varepsilon_n(A)$ in \cite[\S 21]{Mat86}, which is given by the dimension of Koszul homology over $k$, is not equivalent to that of Definition \ref{def:delta} (1). 
\end{remark}

Let $(X, \mathcal{O}_X)$ be a locally Noetherian scheme. For any $x \in X$, we denote by $\mathcal{O}_{X,x}$ the stalk of $\mathcal{O}_X$ at $x$. We can define the functions $D$ and $\Delta_n$ on $X$ using the invariants in Definition \ref{def:delta} as follows. 

\begin{definition}\label{def:delta_scheme}
    Let $(X, \mathcal{O}_X)$ be a locally Noetherian scheme. For any $x \in X$, we define the functions $D$ and $\Delta_n$ on $X$ by
    \begin{align*}
        D &\colon X \rightarrow \NN \colon x \mapsto d(\mathcal{O}_{X,x}) = \dim(\mathcal{O}_{X,x}) - \varepsilon_0(\mathcal{O}_{X,x}) + \varepsilon_1(\mathcal{O}_{X,x}) \\ 
        \Delta_n &\colon X \rightarrow \NN \colon x \mapsto \delta_n(\mathcal{O}_{X,x}) = \dim_{k(x)} H_n(\mathcal{O}_{X,x}, k(x), k(x)),
    \end{align*}
    where $k(x)$ is the residue field of $\mathcal{O}_{X,x}$.
\end{definition}

Ragusa, Tarrio, Rodicio, and Avramov obtained upper semicontinuity results for the functions $D$ and $\Delta_n$ on various schemes, including excellent schemes. The following lemma is easy to check using Lemma \ref{lem:uppersemiconti_opencover} and their results.

\begin{lemma}\label{lem:uppersemiconti}
    Let $(X, \mathcal{O}_X)$ be an excellent scheme \textup{(\cite[(7.8.5)]{Gro65})}. 
    \begin{enumerate}
        \item The function $D$ is upper semicontinuous on $X$; see \cite[THEOREM]{ATR89}.
        \item For any $n>1$, the function $\Delta_n$ is upper semicontinuous on $X$; see \cite[PROPOSITION 3.6]{Rag80}.
    \end{enumerate}
\end{lemma}

It is clear that $d(A)$ and $\delta_n(A)$ are non-negative integers for every Noetherian local ring $A$. Hence, on a quasi-compact excellent scheme, the function $D$ and, for each fixed $n>1$, the function $\Delta_n$ take only finitely many values. In other words, we have the following lemma. 

\begin{lemma}\label{lem:bounded}
    Let $(X, \mathcal{O}_X)$ be a quasi-compact excellent scheme and $n>1$. Then, the functions $D$ and $\Delta_n$ are bounded above on $X$, i.e., there exists $N \in \NN$ such that $D(x), \Delta_n(x) \leq N$ for all $x \in X$. In particular, the functions $D$ and $\Delta_n$ take only finitely many values on $X$.
\end{lemma}
\begin{proof}
    Let $\NN_{\leq m} = \{k \in \NN \mid k \leq m\}$ for any $m \in \NN$. We have
    \begin{align*}
        X = \bigcup_{m \in \NN} D^{-1}(\NN_{\leq m}). 
    \end{align*}
    By Lemma \ref{lem:uppersemiconti}, $D^{-1}(\NN_{\leq m})$ is open in $X$ for any $m \in \NN$. Since $X$ is quasi-compact, there exists $N \in \NN$ such that
    \begin{align*}
        X = \bigcup_{m=0}^{N} D^{-1}(\NN_{\leq m}) = D^{-1}(\NN_{\leq N}).
    \end{align*}
    Hence, $D$ is bounded above on $X$. The same argument shows that $\Delta_n$ is bounded above on $X$. 
\end{proof}

It is natural to ask whether $C \colon x \mapsto \embdim(\mathcal{O}_{X, x}) - \dim(\mathcal{O}_{X, x})$ is upper semicontinuous on $X$ because it measures the codimension of a local ring and, when restricted to the complete intersection locus, distinguishes complete intersections of bounded codimension. However, this function is not upper semicontinuous in general even when $X$ is an affine excellent scheme. For example, let $A = k[[s,t,u]]/((s) \cap (t, u)^2)$ and $X = \Spec(A)$. Then, we have $\embdim(A) - \dim(A) = 3 - 2 = 1$ and $\embdim(A_{(t,u)}) - \dim(A_{(t,u)}) = 2 - 0 = 2$. Hence, the set $C^{-1}(\{k \in \NN \mid k \leq 1\})$ is not stable under generalization. In particular, this set is not open in $X$ by Nagata's topological criterion (Lemma \ref{lem:nagatatop}). 

However, we can claim the following theorem. 

\begin{theorem}\label{th:constructible}
    Let $(X,\mathcal{O}_X)$ be a quasi-compact excellent scheme. Then, for the function
    \begin{align*}
        C \colon X \longrightarrow \NN, \qquad x \longmapsto \embdim(\mathcal{O}_{X,x})-\dim(\mathcal{O}_{X,x})
    \end{align*}
    on $X$, the subset
    \begin{align*}
        C^{-1}(\NN_{\leq n}) = \{x\in X\mid C(x)\leq n\}
    \end{align*}
    is constructible in $X$ for every $n \in \NN$.
\end{theorem}
\begin{proof}
    By Lemma \ref{lem:uppersemiconti}, we have that $D$ and $\Delta_2$ are upper semicontinuous on $X$. Moreover, by Remark \ref{rem:invariants}, we have $C(x) = \Delta_2(x) - D(x)$ for any $x \in X$ (note that this function is not necessarily upper semicontinuous). Hence, for any $n \in \NN$, we have
    \begin{align*}
        C^{-1}(\NN_{\leq n}) &= \{x \in X \mid C(x) \leq n\} \\ 
        &= \bigcup_{r,s \in \NN, r-s \leq n} \{x \in X \mid D(x) = s, \Delta_2(x) = r\} \\ 
        &= \bigcup_{r,s \in \NN, r-s \leq n} (D^{-1}(\NN_{\leq s}) \setminus D^{-1}(\NN_{\leq s-1})) \cap (\Delta_2^{-1}(\NN_{\leq r}) \setminus \Delta_2^{-1}(\NN_{\leq r-1})).
    \end{align*}
    This union is finite by Lemma \ref{lem:bounded}. Moreover, by Lemma \ref{lem:uppersemiconti}, $D^{-1}(\NN_{\leq s})$ and $\Delta_2^{-1}(\NN_{\leq r})$ are open in $X$. Hence, $C^{-1}(\NN_{\leq n})$ is a constructible subset of $X$. Therefore, we have shown that $C^{-1}(\NN_{\leq n})$ is constructible in $X$ for every $n \in \NN$.
\end{proof}

We now obtain the following corollary, which is the main result of this section. Although $\mathsf{CI}_{\leq c}$ does not satisfy (NC) in general for $c \geq 1$, the constructibility of $C$ yields another proof that $\mathsf{CI}_{\leq c}(A)$ is open for every excellent ring $A$. Before proving it, we use the following lemma. 

\begin{lemma}[{\cite[\href{https://stacks.math.columbia.edu/tag/0542}{Tag 0542}]{stacks-project}}]\label{lem:construct}
    Let $X$ be a Noetherian sober topological space. Let $E \subset X$ be a subset of $X$.
    \begin{enumerate}
        \item If $E$ is constructible and stable under specialization, then $E$ is closed. 
        \item If $E$ is constructible and stable under generalization, then $E$ is open.
    \end{enumerate}
\end{lemma}

\begin{corollary}\label{cor:exopen2}
    Let $A$ be an excellent ring. Then, $\mathsf{CI}_{\leq c}(A)$ is an open subset of $\Spec(A)$ for any $c \geq 0$.
\end{corollary}
\begin{proof}
    Let $A$ be an excellent ring. By definition, we have
    \begin{align*}
        \mathsf{CI}_{\leq c}(A) = \mathsf{CI}(A) \cap \{\mathfrak{p} \in \Spec(A) \mid \embdim(A_{\mathfrak{p}}) - \dim(A_{\mathfrak{p}}) \leq c\}.
    \end{align*}
    By Proposition \ref{prop:NCandExcellent} and the fact that $\mathsf{CI}$ satisfies \NC{}, we have that $\mathsf{CI}(A)$ is an open subset of $\Spec(A)$. 

    The set $\{\mathfrak{p} \in \Spec(A) \mid \embdim(A_{\mathfrak{p}}) - \dim(A_{\mathfrak{p}}) \leq c\}$ is a constructible subset of $\Spec(A)$ by Theorem \ref{th:constructible}. Therefore, $\mathsf{CI}_{\leq c}(A)$ is a constructible subset and is stable under generalization. Hence, by Lemma \ref{lem:construct}, $\mathsf{CI}_{\leq c}(A)$ is an open subset of $\Spec(A)$.
\end{proof}

\section{Open questions}\label{sec:openq}

First, we summarize the results of this paper in Table \ref{tab:summary}. Let $c$ be a positive integer.

\begin{table}[H]
    \centering
    \caption{Summary of the properties considered in this paper}\label{tab:summary}
    \small
    \begin{tabular}{
        l
        *{5}{>{\raggedright\arraybackslash}p{1.7cm}}
    }
        \toprule
        $\text{Property}^{(\ref{def:main_sing},\ \ref{def:cicodim})}$ & $\mathsf{Reg}$ & $\mathsf{CI}_{\leq c}$ & $\mathsf{CI}$ & $\mathsf{Gor}$ & $\mathsf{CM}$ \\
        \midrule
        $\text{(A-1)}^{(\ref{def:conditions-for-NC})}$ & $\checkmark$ & $\checkmark^{(\ref{prop:nagata_hs})}$ & $\checkmark$ & $\checkmark$ & $\checkmark$ \\
        $\text{(A-2)}^{(\ref{def:conditions-for-NC})}$ & $\checkmark$ & $\checkmark^{(\ref{prop:nagata_hs})}$ & $\checkmark$ & $\checkmark$ & $\checkmark$ \\
        $\text{(A-3)}^{(\ref{def:conditions-for-NC})}$ & $\checkmark^{(\ref{prop:known_singularity_NC})}$ & $\times^{(\ref{th:counterexampleNC})}$ & $\checkmark^{(\ref{prop:known_singularity_NC})}$ & $\checkmark^{(\ref{prop:known_singularity_NC})}$ & $\checkmark^{(\ref{prop:known_singularity_NC})}$ \\
        $\text{\NC{}}^{(\ref{def:Nagata_criterion})}$ & $\checkmark$ & $\times^{(\ref{th:counterexampleNC})}$ & $\checkmark$ & $\checkmark$ & $\checkmark$ \\
        $\text{\QC{}}^{(\ref{def:quotient_condition})}$ & $\times^{(\ref{rem:regular_QC})}$ & $?^{(\ref{q:QC_cifin})}$ & $\checkmark^{(\ref{rem:regular_QC})}$ & $\checkmark^{(\ref{rem:regular_QC})}$ & $\checkmark^{(\ref{rem:regular_QC})}$ \\
        Polynomial extensions & $\checkmark$ & $\checkmark^{(\ref{prop:polyext})}$ & $\checkmark$ & $\checkmark$ & $\checkmark$ \\
        Openness for $\mathsf{Reg}$-Q0 rings & $\checkmark $ & $\checkmark^{(\ref{th:openness_hs})}$ & $\checkmark $ & $\checkmark $ & $\checkmark $ \\
        \bottomrule
    \end{tabular}

    \medskip

    \begin{minipage}{0.92\textwidth}
        \footnotesize Here, $\checkmark$ means that the property holds, $\times$ means that it does not hold in general, and $?$ means that it is an open question. Superscript numbers refer to the corresponding results in this paper. Checkmarks without superscript numbers are well-known or immediate from other results.
    \end{minipage}
\end{table}

This leads to the following open question.

\begin{question}\label{q:QC_cifin}
    Does $\mathsf{CI}_{\leq c}$ satisfy \QC{} in general for any $c \geq 1$?
\end{question}

Valabrega noted in \cite{Val78} that it was not known whether regularity satisfies \QC{} in characteristic $0$. More explicitly, let $A$ be a regular ring (not necessarily local) of characteristic $0$. Then, does $\mathsf{Reg}(A/\mathfrak{p})$ contain a nonempty open subset of $\Spec(A/\mathfrak{p})$ for every $\mathfrak{p} \in \Spec(A)$? To the best of the author’s knowledge, this question remains open.

We proceed to the next question. In this paper, we have mainly discussed $\mathsf{CI}_{\leq c}$. There are several other classes of singularities that one may consider. 

\begin{definition}\label{def:other_sing}
    Let $(A, \mathfrak{m}, k)$ be a Noetherian local ring. We say that $A$ is
    \begin{enumerate}
        \item nearly Gorenstein ($\mathsf{nGor}$) if $A$ is Cohen--Macaulay and has a canonical module $\omega_A$ such that $\mathfrak{m} \subset \tr_A(\omega_A)$, where $\tr_A(\omega_A)$ is the trace ideal of $\omega_A$,
        \item canonical trace radical ($\mathsf{CTR}$) if $A$ is Cohen--Macaulay and has a canonical module $\omega_A$ such that $\tr_A(\omega_A)$ is a radical ideal, where $\tr_A(\omega_A)$ is the trace ideal of $\omega_A$,
        \item Cohen--Macaulay with canonical module ($\mathsf{CMC}$) if $A$ is Cohen--Macaulay and has a canonical module $\omega_A$. 
    \end{enumerate}
\end{definition}

Nearly Gorenstein rings were introduced in \cite[Definition 2.2]{HHS19}, and canonical trace radical rings were introduced in \cite[Definition 3.2]{Miy26}. These classes satisfy the following hierarchy: 
\begin{align*}
    \mathsf{Reg} \Rightarrow \mathsf{CI}_{\leq c} \Rightarrow \mathsf{CI} \Rightarrow \mathsf{Gor} \Rightarrow \mathsf{nGor} \Rightarrow \mathsf{CTR} \Rightarrow \mathsf{CMC} \Rightarrow \mathsf{CM}.
\end{align*}

This leads to the following natural question. 

\begin{question}\label{q:other_sing}
    Do $\mathsf{nGor}$, $\mathsf{CTR}$, and $\mathsf{CMC}$ satisfy \NC{} or \QC{}?
\end{question}

We will need to study the behavior of canonical modules for these singularities to answer Question \ref{q:other_sing}. We leave this for future work.

\bibliographystyle{amsalpha-inits}
\bibliography{references}

\providecommand{\bysame}{\leavevmode\hbox to3em{\hrulefill}\thinspace}
\providecommand{\MR}{\relax\ifhmode\unskip\space\fi MR }
\providecommand{\MRhref}[2]{%
  \href{http://www.ams.org/mathscinet-getitem?mr=#1}{#2}
}
\providecommand{\href}[2]{#2}
\begin{thebibliography}{NSW15}

\bibitem[AM69]{AM69}
M.~F. Atiyah and I.~G. Macdonald, \emph{Introduction to commutative algebra}, Addison-Wesley Publishing Co., Reading, MA, 1969. \MR{242802}

\bibitem[ATR89]{ATR89}
L.~Alonso~Tarrio and A.~G. Rodicio, \emph{On the upper semicontinuity intersection defect}, Math. Scand. \textbf{65} (1989), no.~2, 161--164. \MR{1050861}

\bibitem[Avr77]{Avr77}
L.~L. Avramov, \emph{Homology of local flat extensions and complete intersection defects}, Math. Ann. \textbf{228} (1977), no.~1, 27--37. \MR{485836}

\bibitem[Avr98]{Avr98}
\bysame, \emph{Infinite free resolutions}, Six lectures on commutative algebra ({Bellaterra}, 1996), Progress in Mathematics, vol. 166, Birkh{\"a}user, Basel, 1998, pp.~1--118. \MR{1648664}

\bibitem[GL69]{GL69}
T.~H. Gulliksen and G.~Levin, \emph{Homology of local rings}, Queen's Papers in Pure and Applied Mathematics, vol.~20, Queen's University, Kingston, ON, 1969. \MR{262227}

\bibitem[GM78]{GM78}
S.~Greco and M.~G. Marinari, \emph{{Nagata}'s criterion and openness of loci for {Gorenstein} and complete intersection}, Math. Z. \textbf{160} (1978), no.~3, 207--216. \MR{491741}

\bibitem[Gro65]{Gro65}
A.~Grothendieck, \emph{{\'E}l{\'e}ments de g{\'e}om{\'e}trie alg{\'e}brique {IV}: {\'E}tude locale des sch{\'e}mas et des morphismes de sch{\'e}mas, seconde partie}, Publ. Math. Inst. Hautes {\'E}tudes Sci. \textbf{24} (1965), 5--231.

\bibitem[HHS19]{HHS19}
J.~Herzog, T.~Hibi, and D.~I. Stamate, \emph{The trace of the canonical module}, Israel J. Math. \textbf{233} (2019), no.~1, 133--165. \MR{4013970}

\bibitem[Hir64]{Hir64}
H.~Hironaka, \emph{Resolution of singularities of an algebraic variety over a field of characteristic zero. {I}, {II}}, Ann. of Math. (2) \textbf{79} (1964), 109--203; 205--326. \MR{199184}

\bibitem[Kan91]{Kan91}
T.~Kanzo, \emph{{Nagata} criterion for {Buchsbaum} loci}, Math. J. Toyama Univ. \textbf{14} (1991), 177--184. \MR{1145153}

\bibitem[Kim23]{Kim23}
K.~Kimura, \emph{Openness of various loci over {Noetherian} rings}, J. Algebra \textbf{633} (2023), 403--424.

\bibitem[Mat86]{Mat86}
H.~Matsumura, \emph{Commutative ring theory}, Cambridge Studies in Advanced Mathematics, vol.~8, Cambridge University Press, Cambridge, 1986. \MR{879273}

\bibitem[Miy26]{Miy26}
M.~Miyazaki, \emph{Radical property of the traces of the canonical modules of {Cohen--Macaulay} rings}, J. Math. Soc. Japan (2026), 1--24, Advance publication.

\bibitem[MV77]{MV77}
C.~Massaza and P.~Valabrega, \emph{Sull'apertura di luoghi in uno schema localmente noetheriano}, Boll. Un. Mat. Ital. A (5) \textbf{14} (1977), no.~3, 564--574. \MR{480501}

\bibitem[Nag59]{Nag59}
M.~Nagata, \emph{On the closedness of singular loci}, Publ. Math. Inst. Hautes {\'E}tudes Sci. \textbf{2} (1959), 5--12. \MR{106908}

\bibitem[Nis12]{Nis12}
J.~Nishimura, \emph{A few examples of local rings, {I}}, Kyoto J. Math. \textbf{52} (2012), no.~1, 51--87. \MR{2892767}

\bibitem[NSW15]{NSW15}
S.~Nasseh and S.~Sather-Wagstaff, \emph{Cohen factorizations: weak functoriality and applications}, J. Pure Appl. Algebra \textbf{219} (2015), no.~3, 622--645. \MR{3279378}

\bibitem[Put21]{Put21}
T.~J. Puthenpurakal, \emph{Localization of complete intersections}, Comm. Algebra \textbf{49} (2021), no.~10, 4543--4545. \MR{4296856}

\bibitem[Rag80]{Rag80}
A.~Ragusa, \emph{On openness of {$H_n$}-locus and semicontinuity of {$n$}th deviation}, Proc. Amer. Math. Soc. \textbf{80} (1980), no.~2, 201--209. \MR{577744}

\bibitem[Sta26]{stacks-project}
{The Stacks Project Authors}, \emph{The {Stacks Project}}, \url{https://stacks.math.columbia.edu}, 2026.

\bibitem[Tak99]{Tak99}
R.~Takahashi, \emph{{Nagata} criterion for {Serre}'s {$(R_n)$} and {$(S_n)$}-conditions}, Math. J. Okayama Univ. \textbf{41} (1999), no.~1, 37--43. \MR{1816615}

\bibitem[Val78]{Val78}
P.~Valabrega, \emph{Formal fibers and openness of loci}, J. Math. Kyoto Univ. \textbf{18} (1978), no.~1, 199--208. \MR{485865}

\end{thebibliography}

\end{document}